\documentclass[12pt]{article}
\usepackage{amsmath}
\usepackage{amsfonts}
\usepackage{amsthm}
\usepackage{graphicx}
\usepackage{overpic}
\usepackage{a4wide}
\usepackage{amssymb} 
\usepackage{pictex}
\usepackage{rotating}  
\usepackage{color}
\usepackage{etex} 
\usepackage{enumitem}
\usepackage{accents}

\definecolor{refkey}{rgb}{0,0,1}
\definecolor{labelkey}{rgb}{1,0,0}

\usepackage{comment}

\numberwithin{equation}{section}

\newtheorem{theorem}{Theorem}[section]
\newtheorem{proposition}[theorem]{Proposition}
\newtheorem{lemma}[theorem]{Lemma}
\newtheorem{corollary}[theorem]{Corollary}
\newtheorem{Definition}[theorem]{Definition}
\newenvironment{definition}{\begin{Definition}\rm}{\end{Definition}}
\newtheorem{Remark}[theorem]{Remark}
\newenvironment{remark}{\begin{Remark}\rm}{\end{Remark}}
\newtheorem{Example}[theorem]{Example}
\newenvironment{example}{\begin{Example}\rm}{\end{Example}}
\newtheorem{RHproblem}[theorem]{RH problem}

\usepackage{color}

\newcommand{\C}{\mathbb{C}}

\newcommand{\R}{\mathbb{R}}

\renewcommand{\AA}{\mathcal A}

\newcommand{\KK}{\mathcal K}

\newcommand{\MM}{\mathcal M}

\newcommand{\OO}{\mathcal{O}}

\renewcommand{\Re}{{\rm Re} \,}

\renewcommand{\hat}{\widehat}
\renewcommand{\tilde}{\widetilde}

\def\supp{\mathop{\mathrm{supp}}\nolimits}

\usepackage[bookmarksopen, naturalnames]{hyperref}

\begin{document}
\title{Modified logarithmic potential theory and applications}
\author{Thomas Bloom, Norman Levenberg, Vilmos Totik and Franck Wielonsky}
\maketitle 
%
%
%
\begin{abstract} We develop potential theory including a Bernstein-Walsh type estimate for
functions of the form $p(z)q(f(z))$ where $p,q$ are polynomials and $f$ is
holomorphic. Such functions arise in the study of certain ensembles of
probability measures and our estimates lead to probabilistic results such as large deviation principles. \end{abstract}

\medskip

\section{Introduction}
The classical Bernstein-Walsh inequality establishes growth rates for polynomials $p$ outside of a compact set $K\subset \C$ in terms of the supremum norm of $p$ on $K$ and the degree of $p$:
$$|p(z)|\leq \bigl(\sup_{\zeta \in K}|p(\zeta)|\bigr)e^{deg(p)V_K(z)}=:||p||_Ke^{deg(p)V_K(z)}$$
where $V_K$ is the extremal function for $K$ (see (\ref{forbw})). Given a finite measure $\mu$ on $K$, a Bernstein-Markov type inequality is a comparability between $L^p(\mu)$ norms ($1<p<\infty$) and supremum norms for polynomials of a given degree:
$$||p||_K\leq M_k||p||_{L^p(\mu)} \ \hbox{for polynomials of degree} \ k$$
where $M_k^{1/k}\to 1$. In a series of papers in various potential-theoretic settings (cf., \cite{VELD} and \cite{LPLD}), the authors have studied analogues of these properties. With these inequalities established, using purely potential-theoretic techniques one can prove probabilistic results such as large deviation principles associated to empirical distributions arising from discretizing the associated potential-theoretic energy minimization problem. An essential ingredient is the weighted version of the problem. 

In this paper, given $K\subset \C$, we consider the problem of minimizing the weighted energy 
$$E^Q_f(\mu)=E^Q(\mu):= \int_{K}\int_{K}\log\frac{1}{|x-y||f(x)-f(y)|w(x)w(y)}d\mu(x)d\mu(y)$$
over probability measures $\mu $ on $K$ where $w=e^{-Q}$ is a weight function on $K$ and $f:K\to \C$ is a fixed function. Discretizing the problem, for each $k=1,2,...$, we consider maximizing the weighted  $f-$Vandermonde of order $k$:
$$|VDM_k^Q(z_0,...,z_k)|:=$$
$$|VDM(z_0,...,z_k)|\exp \Big(-k[Q(z_0)+\cdots + Q(z_k)]\Big)|VDM(f(z_0),...,f(z_k))|$$ over $k+1$ tuples of points $z_0,...,z_k\in K$ where
$VDM(z_0,...,z_k)=\prod_{0\leq  i<j\leq k}(z_j-z_i)$
 is the classical Vandermonde determinant. After developing the potential-theoretic background for appropriate $K, Q$ and $f$ in sections 2 and 3, we obtain Bernstein-Walsh type estimates for the ``generalized weighted $f-$polynomials'' 
 $$z_j \to VDM_k^Q(z_0,...,z_k)$$ 
 where $f$ is holomorphic on a neighborhood of $K$. This is a special case of the more general estimates (\ref{BW}) and (\ref{BWQ}) in section 4 for functions of the form
 $$h_k(z)=p_k(g(z))q_k(f(z)), \ p_k, q_k \ \hbox{polynomials of degree} \ k$$
 where $f,g$ are defined and holomorphic on a neighborhood of $K$. 
 
Following standard arguments (cf., \cite{BLmass}), given a measure $\nu$ on $K$ satisfying a mass-density condition, it follows that the $k(k+1)/2$ roots of the averages 
 $$Z_k:=\int_{K^{k+1}} |VDM_k^Q(z_0,...,z_{k})| d\nu(z_0) \cdots d\nu(z_{k})$$ 
 tend to the same limit as the $k(k+1)/2$ roots of the maximal weighted  $f-$Vandermondes $|VDM_k^Q(z_0,...,z_k)|$ over $K^{k+1}$. This has consequences for the empirical distribution associated to the ensemble of probability measures $Prob_k$ on $K^{k+1}$, where, for a Borel set $A\subset K^{k+1}$,
$$Prob_k(A):=\frac{1}{Z_k}\cdot \int_A |VDM_k^Q(z_0,...,z_k)|d\nu (z_0)\cdots d\nu (z_k).$$
These consequences are the main content of section 5, where we restrict to compact $K$. The brief section 6 details the key ingredients needed to make extensions to the unbounded case.

There are numerous articles in the literature where various aspects of the ensembles considered in this paper are studied. For $f(z)=e^z$ and $K=\R$ see Claeys-Wang \cite{CW}. For $f(z)=z^{\theta}, \ \theta >0$ and $K=\R^+$ they were studied by Borodin \cite{Bo}. He named them {\it biorthogonal ensembles}. For $\theta =2$ they were studied in Leuck, Sommers and Zirnbauer \cite{LSZ} motivated by physical considerations. For $\theta$ a positive integer, a large deviation result was proved by Eichelsbacher, Sommerauer and Stolz in \cite{ESS} under some restrictions on $Q$. 

Recent papers of Cheliotis \cite{CH} and Forrester-Wang \cite{FW} exhibit these ensembles as joint probability distributions of eigenvalues of specific ensembles of random matrices. The case $f(z)=\log z$ also occurs this way.

Work of Muttalib \cite{Mu} originally provided impetus for studying these ensembles. He had proposed a correction term to the joint probability distribution of the GUE (Gaussian unitary ensemble) to describe certain physical phenomena. In particular, he proposed to consider $f(z)=\log ({\rm arcsinh} ^2 z^{1/2})$ on $\R^+$.

A paper of Chafai, Gozlan and Zitt \cite{CGZ} establishes a large deviation principle on $\R^d$ under quite general circumstances. Restricted to $\R^2$ or $\R$, there is some overlap with the probabilistic results of this paper.

\section{General potential theory results}
In this section we state and prove results, including existence and uniqueness of weighted energy minimizing measures, in a univariate setting generalizing the classical setting in \cite{ST} (see also \cite{LSS} for a particular case). Recall a set $E\subset \C$ is {\it polar} if there exists $u\not \equiv -\infty$ defined and subharmonic on a neighborhood of $E$ with $E\subset \{u=-\infty\}$ (cf., \cite{ST}). We use the terminology that a property holds q.e. (quasi-everywhere) on a set $S\subset \C$ if it holds on $S\setminus P$ where $P$ is a polar set. In \cite{ST}, given a compact, nonpolar set $K\subset \C$, a real-valued function $Q$ on $K$ is called {\it admissible} if $Q$ is lower semicontinuous and $\{z\in K: Q(z) <\infty\}$ is not polar. We write $Q\in \AA(K)$ and define $w(z):=e^{-Q(z)}$. If $K$ is closed but unbounded, one requires that
\begin{equation}\label{defadm}\liminf_{|z|\to \infty, \ z\in K}[Q(z)- \frac{1}{2}\log (1+|z|^2)]=\infty.\end{equation}   

Suppose now a closed, nonpolar set $K\subset \C$ is given, and $f:K\to \C$ is continuous. For $K$ compact, the class of admissible weights $Q$ on $K$ suffices for our purposes; for unbounded $K$, we make the following definition. 

  \begin{definition}\label{admit} We call a lower semicontinuous function $Q$ on a closed, unbounded set $K\subset \C$ with $\{z\in K: Q(z) <\infty\}$ not polar {\it $f-$admissible} for $K$ if
  $$\psi(z):= Q(x)-\frac{1}{2}\log {[(1+|z|^2)(1+|f(z)|^2)]}$$ 
satisfies $\lim_{|z|\to \infty, \ z\in K}\psi(z)=\infty$.
\end{definition}

Note that this implies $\psi(z)\geq c =c(Q)>-\infty$ for all $z\in K$; also, since $1+|f(z)|^2\geq 1$, we have $\psi(z)\leq Q(z)-\frac{1}{2}\log (1+|z|^2)$ so that $Q$ is admissible in the usual potential-theoretic sense (\ref{defadm}) of \cite{ST}. The hypothesized growth of $Q$ depends heavily on $f$. We say $Q$ is {\it strongly $f-$admissible} for $K$ if there exists $\delta >0$ such that $(1-\delta)Q$ is $f-$admissible for $K$. 

The weighted potential theory problem we study is to minimize the weighted energy
\begin{equation}\label{def-EQ}
E^Q_f(\mu)=E^Q(\mu):= \int_{K}\int_{K}\log\frac{1}{|x-y||f(x)-f(y)|w(x)w(y)}d\mu(x)d\mu(y)
\end{equation}
over $\mu \in \mathcal M(K)$, the set of probability measures on $K$. Here $w=e^{-Q}$.
Note that the double integral in (\ref{def-EQ}) is well-defined and different from $-\infty$. Indeed, let
\begin{equation}\label{def-k}
k(x,y):=-\log \left(|x-y||f(x)-f(y)|w(x)w(y)\right).
\end{equation}
Using the inequality
$|u-v|\leq \sqrt {1+|u|^2} \sqrt {1+|v|^2}$,
we have
\begin{multline*}
\log {|x-y|}+ \log {|f(x)-f(y)|} \\ \leq
\frac{1}{2}\log {(1+|x|^2)} +\frac{1}{2}\log {(1+|y|^2)} +\frac{1}{2}\log {(1+|f(x)|^2)} 
+\frac{1}{2}\log {(1+|f(y)|^2)}.
\end{multline*}
Hence, by Definition \ref{admit}, 
\begin{equation}\label{ineg-k-psi}
k(x,y)\geq \psi(x) +\psi(y)\geq 2c \ \hbox{on} \ K\times K,
\end{equation}
and the integrand of the double integral is bounded below by $2c$. 

We also recall the definition of the logarithmic energy of $\mu$,
$$I(\mu):=\int_K \int_K \log \frac{1}{|x-y|}d\mu(x) d\mu(y)=:\int_K p_{\mu}(y)d\mu(y)$$
where $p_{\mu}(y):=  \int_K \log \frac{1}{|x-y|}d\mu(x)$ is the logarithmic potential of $\mu$. For $K\subset \C$ compact, the logarithmic capacity of $K$ is 
\begin{equation}\label{capdef} cap(K):=\exp \bigl[-\inf\{I(\mu): \mu\in \mathcal M(K)\}\bigr].\end{equation}  
For a Borel set $E\subset \C$, $cap(E)$ may be defined as $\exp \bigl[-\inf I(\mu)]$ where the infimum is taken over all Borel probability measures with compact support in $E$. The weighted logarithmic energy of $\mu$ with respect to $Q$ is 
\begin{equation}\label{def-IQ}
I^Q(\mu):=\int_{K}\int_{K}\log\frac{1}{|x-y|w(x)w(y)}d\mu(x)d\mu(y).
\end{equation}
Since $1+|f(x)|^{2}\geq 1$, the double integral in (\ref{def-IQ}) is also well-defined and different from $-\infty$. When $I(\mu)\neq-\infty$ or $\int Qd\mu<\infty$, we can rewrite $I^{Q}(\mu)$ as
$$I^{Q}(\mu)=I(\mu)+2\int_K Qd\mu.$$
For the push-forward measure $f_*\mu$ of $\mu$ on $f(K)$, we have
\begin{align*}
I(f_*\mu) & =\int_K \int_K \log \frac{1}{|f(x)-f(y)|}d\mu(x) d\mu(y)
=\int_{f(K)}\int_{f(K)}\log \frac{1}{|a-b|}df_*\mu(a) df_*\mu(b)
\\
& =\int_{f(K)}p_{f_*\mu}(b) df_*\mu(b)=\int_K p_{f_*\mu}(f(z))d\mu(z).
\end{align*}
When $I^{Q}(\mu)\neq+\infty$ or $I(f_{*}\mu)\neq-\infty$, the energy $E^{Q}(\mu)$ can be rewritten as
$$E^{Q}(\mu)=I^Q(\mu) + I(f_*\mu).$$
  \begin{proposition}\label{exun}
Let $K\subset \C$ be closed and let $Q$ be $f-$admissible for $K$. Suppose there exists $\nu \in \mathcal M(K)$ with $E^Q(\nu)<\infty$. Let
$V_{w}:=\inf\{E^{Q}(\mu),~\mu\in\MM(K)\}$.
Then 
\begin{enumerate}
\item $V_{w}$ is finite.
\item Setting $K_{M}:=\{z: Q(z)\leq M\}$, we have, for sufficiently large $M<\infty$,
$$V_{w}=\inf\{E^{Q}(\mu),~\mu\in\MM(K_{M})\}.$$
\item We have existence and uniqueness of $\mu_{K,Q}$ minimizing $E^Q$. The measure $\mu_{K,Q}$ has compact support and the logarithmic energies $I(\mu_{K,Q})$ and $I(f_{*}\mu_{K,Q})$ are finite. 
\item
The following Frostman-type inequalities hold true:
\begin{align}\label{Frost1}
p_{\mu_{K,Q}}(z)  +p_{f_*\mu_{K,Q}}(f(z))+Q(z) & \geq F_{w} \ \hbox{q.e. on}  \ K,\\
p_{\mu_{K,Q}}(z)  +p_{f_*\mu_{K,Q}}(f(z))+Q(z) & \leq F_{w} \ \hbox{on supp}(\mu_{K,Q}),\label{Frost2}
\end{align}
where $F_{w}:=I(\mu_{K,Q})+I(f_{*}\mu_{K,Q})+\int Qd\mu_{K,Q}
=V_{w}-\int Qd\mu_{K,Q}$.
\item if a measure $\mu\in \mathcal M(K)$ with compact support and $E^{Q}(\mu)<\infty$ satisfies 
\begin{align}\label{frost1}
p_{\mu}(z)  +p_{f_*\mu}(f(z))+Q(z) & \geq C \ \hbox{q.e. on}  \ K,\\
p_{\mu}(z)  +p_{f_*\mu}(f(z))+Q(z) & \leq C \ \hbox{on supp}(\mu),\label{frost2}
\end{align}
for some constant $C$, then $\mu =\mu_{K,Q}$.
\end{enumerate}
      \end{proposition}
      
  \begin{proof} For 1., we have $V_{w}<\infty$ by assumption. The other inequality 
$-\infty<V_{w}$ follows from the fact that the double integral in (\ref{def-EQ}) is bounded below by $2c$. The proof of 2. follows the lines of \cite[p. 29-30]{ST}, namely one first proves that, for $M$ sufficiently large,
$$k(x,y)>V_{w}+1\quad\text{ if }\quad(x,y)\notin K_{M}\times K_{M},$$
from which one derives that $E^{Q}(\mu)=V_{w}$ is possible only for measures with support in $K_{M}$. 

We next prove 3. From 2., there is a sequence
$\{\mu_n\}\subset \mathcal M(K_{M})$ with 
$$E^Q(\mu_n)\to V_{w}\quad\text{as}\quad n\to\infty.$$
The set $K_{M}$ is compact, hence, by Helly's theorem, we get a subsequence of these measures converging weakly to a probability measure $\mu$ supported on 
$K_{M}$; and it is easy to see this $\mu:=\mu_{K,Q}$ satisfies $E^Q(\mu)=V_{w}$. 
For the logarithmic energy of $\mu_{K,Q}$, we have $I(\mu_{K,Q})>-\infty$ because $\mu_{K,Q}$ has compact support. Since $f$ is continuous and $f_{*}\mu_{K,Q}$ has its support in $f(K_{M})$, we also have $I(f_{*}\mu_{K,Q})>-\infty$. Now, recalling that $Q$ is bounded below, we may write $I(\mu_{K,Q})$ as the well-defined expression
$$I(\mu_{K,Q})=V_{w}-I(f_{*}\mu_{K,Q})-2\int_{K}Qd\mu_{K,Q},$$
from which follows that $I(\mu_{K,Q})<\infty$ and then also 
$I(f_{*}\mu_{K,Q})<\infty$.

The uniqueness follows from the fact that $\mu \to  I(\mu)$ is strictly convex and $\mu \to I(f_*\mu)$ is convex on the subsets of $\mathcal M(K)$ where they are finite. To be precise, it is well-known that for $\mu_{1}$ and $\mu_{2}$ two measures with finite energies and $\mu_1(K)=\mu_2(K)$, we have $I(\mu_{1}-\mu_{2})\geq 0$ and $I(\mu_{1}-\mu_{2})=0$ if and only if $\mu_{1}=\mu_{2}$ (cf., Lemma I.1.8 in \cite{ST}).

Now if $\bar\mu\in \mathcal M(K)$ is another measure which minimizes $E^Q$, we know from the proof of 2. that $\bar\mu\in \mathcal M(K_{M})$. Consequently, 
$I(\bar\mu), I(f_{*}\bar\mu)>-\infty$ and then also $I(\bar\mu), I(f_{*}\bar\mu)<\infty$. We have
$$E^Q(\frac{1}{2}(\mu_{K,Q}+\bar\mu))+I(\frac{1}{2}(\mu_{K,Q}-\bar\mu))
+I(f_*(\frac{1}{2}(\mu_{K,Q}-\bar\mu))=\frac{1}{2}[E^Q(\mu_{K,Q})+E^Q(\bar\mu)]
=V_{w}.$$
The sum $I(\frac{1}{2}(\mu_{K,Q}-\bar\mu))
+I(f_*(\frac{1}{2}(\mu_{K,Q}-\bar\mu))\geq 0$ with equality if and only if $\mu_{K,Q}=\bar\mu$; hence the result. 

We next prove the first inequality in 4. Let $\mu\in\MM(K)$ with compact support and consider the measure $\tilde\mu=t\mu+(1-t)\mu_{K,Q}$, $t\in[0,1]$. The inequality $E^{Q}(\mu_{K,Q})\leq E^{Q}(\tilde\mu)$ can be rewritten as  
\begin{multline*}
E^{Q}(\mu_{K,Q})\leq t^{2}(I(\mu)+I(f_{*}\mu))
+(1-t)^{2}(I(\mu_{K,Q})+I(f_{*}\mu_{K,Q}))\\
+2t(1-t)(I(\mu,\mu_{K,Q})+I(f_{*}\mu,f_{*}\mu_{K,Q}))
+2\int Qd(t\mu+(1-t)\mu_{K,Q}),
\end{multline*}
where, for two measures $\mu$ and $\nu$, we denote by $I(\mu,\nu)$ the mutual logarithmic energy
$$I(\mu,\nu)=-\iint\log|x-y|d\mu(x)d\nu(y).$$
Note that the right-hand side of the above inequality is well-defined since the assumption that $\mu$ has compact support implies that all terms in the sum are larger than $-\infty$. Letting $t$ tend to 0, we obtain
\begin{equation}\label{Fw-less}
F_{w}=I(\mu_{K,Q})+I(f_{*}\mu_{K,Q})+\int Qd\mu_{K,Q}\leq 
I(\mu,\mu_{K,Q})+I(f_{*}\mu,f_{*}\mu_{K,Q})
+\int Qd\mu.
\end{equation}
Now, we proceed by contradiction, assuming that there exists a nonpolar compact subset $\KK$ of $K$ such that
$$\forall z\in\KK,\quad p_{\mu_{K,Q}}(z)  
+p_{f_*\mu_{K,Q}}(f(z))+Q(z) < F_{w}.$$ 
Integrating this inequality with respect to a probability measure $\mu$ supported on $\KK$, we obtain
$$I(\mu,\mu_{K,Q})+I(f_{*}\mu,f_{*}\mu_{K,Q})
+\int Qd\mu<F_{w},$$
which contradicts (\ref{Fw-less}). 

The proof of the second inequality in 4. is also by contradiction. Assume that
$$\exists x_{0}\in\supp(\mu_{K,Q}),\quad
p_{\mu_{K,Q}}(x_{0})  +p_{f_*\mu_{K,Q}}(f(x_{0}))+Q(x_{0}) > F_{w}.$$
By lower semicontinuity, the inequality is satisfied in a neighborhood $V_{x_{0}}$ of $x_{0}$. Moreover $\mu_{K,Q}(V_{x_{0}})>0$ since $x_{0}\in\supp(\mu_{K,Q})$. Using the first inequality (\ref{Frost1}) on $\supp(\mu_{K,Q})\setminus V_{x_{0}}$ and the fact that $\mu_{K,Q}(E)=0$ for $E$ a polar set (since $\mu_{K,Q} $ has finite logarithmic energy $I(\mu_{K,Q})$), we obtain
\begin{align*}
F_{w} & =\int (p_{\mu_{K,Q}}(z)  
+p_{f_*\mu_{K,Q}}(f(z))+Q(z))d\mu_{K,Q}(z) \\
& >F_{w}\mu_{K,Q}(V_{x_{0}})+
F_{w}\mu_{K,Q}(\supp(\mu_{K,Q})\setminus V_{x_{0}})=F_{w},
\end{align*}
which is a contradiction.

Finally, we prove 5. We write 
$$\mu_{K,Q}=\mu+(\mu_{K,Q}-\mu).$$ 
Then
$$E^Q(\mu)\geq E^Q(\mu_{K,Q}) = E^Q(\mu) +I(\mu_{K,Q}-\mu)+ I(f_*(\mu_{K,Q}-\mu))+2R$$
with
\begin{align*}
R := & \int_K\bigl[\int_K -\log {|x-y|}d\mu(y) +Q(x)\bigr]d(\mu_{K,Q}-\mu)(x) 
\\  & 
- \int_K\int_K \log {|f(x)-f(y)|}d\mu(y)d(\mu_{K,Q}-\mu)(x)
\\
 = & \int_K(p_{\mu}(x)+Q(x))d(\mu_{K,Q}-\mu)(x)+\int_Kp_{f_*\mu}(f(x))d(\mu_{K,Q}-\mu)(x)
\\
 = & \int_K(p_\mu(x)  +p_{f_*\mu}(f(x))+Q(x))d(\mu_{K,Q}-\mu)(x).
\end{align*}
Note that the above computation is justified. Indeed, from the assumptions $E^{Q}(\mu)<\infty$ and $\mu$ has compact support, the quantities $E^{Q}(\mu)$, $I^{Q}(\mu)$, $I(f_{*}\mu)$, $I(\mu)$, $\int Qd\mu$, and $I(\mu,\mu_{K,Q})$
are all finite. Making use of
the inequalities (\ref{frost1}) and (\ref{frost2}), we derive
$$R\geq  C\int_K d\mu_{K,Q}-C\int_K d\mu=0.$$
Now, recall that $I(\mu_{K,Q}-\mu)+ I(f_*(\mu_{K,Q}-\mu))\geq 0$ with equality if and only if $\mu_{K,Q}=\mu$. Thus 
$$E^Q(\mu) \geq E^Q(\mu_{K,Q}) 
\geq E^Q(\mu)$$
so that equality holds throughout, and $E^Q(\mu) = E^Q(\mu_{K,Q})$, from which
follows $\mu =\mu_{K,Q}$. 
\end{proof}

The condition that there exist $\nu \in \mathcal M(K)$ with $E^Q(\nu)<\infty$ is not automatic. For example, if $f$ is a constant function, then trivially all measures $\nu$ have $I(f_*\nu)=\infty$. We give a sufficient condition on $f$ ensuring the hypothesis of Proposition \ref{exun}.

\begin{proposition}\label{cond-f} 
  If $f:K\to \C$ is continuous and 
  $$\Sigma:= \left\{z\in K:Q(z)<\infty \ \hbox{and} \ 
  \liminf_{\substack{(z_1,z_2)\to (z,z)\\z_1,z_2\in K,~z_{1}\neq z_{2}}}
  \left|\frac{f(z_1)-f(z_2)}{z_1-z_2}\right|>0\right\}$$
  is not polar, then there exist $\nu \in \mathcal M(K)$ with $E^Q(\nu)<\infty$.
  
  \end{proposition}
  
  \begin{proof} Let $D:=\{(z,z):z\in K\}$. Define 
  $$\phi(z_1,z_2):=\left|\frac{f(z_1)-f(z_2)}{z_1-z_2}\right|;$$
  this is continuous on $(K\times K) \setminus D$. Extend $\phi$ to $D$ by defining
  $$\phi(z,z):= \liminf_{\substack{(z_1,z_2)\to (z,z)\\ z_1,z_2\in K,~z_{1}\neq z_{2}}}
  \left|\frac{f(z_1)-f(z_2)}{z_1-z_2}\right|.$$
  Then $\phi:K\times K \to \C$ is lower semicontinuous and we can write 
  $\Sigma= \cup_{n=1}^{\infty} \Sigma_n$
  where 
  $$\Sigma_n:= \{z\in K: Q(z) < n \ \hbox{and} \ \phi(z,z)>1/n\}.$$
  This is an increasing union so for all sufficiently large $n$, $\Sigma_n$ is not polar. Fix such an $n$.
Since polarity is a local property, see e.g. \cite[Remark 4.2.13]{HEL}, there exists $z\in\Sigma_{n}$ such that, for any neighborhood $V_{z}$ of $z$, $\Sigma_{n}\cap V_{z}$ is not polar. 

Now, the function $\phi$ is lower semicontinuous on $K^{2}$, hence there exists a neigborhood $V_{z}$ of $z$ such that 
$\phi(z_{1},z_{2})>1/n$ on $(\Sigma_{n}\cap V_{z})^{2}$ and by the preceeding remark, $\Sigma_{n}\cap V_{z}$ is not polar. Being not polar, $\Sigma_{n}\cap V_{z}$ supports a measure $\nu$ of finite logarithmic energy which is also of finite weighted logarithmic energy since $Q(z)<n$ for $z\in\Sigma_{n}$. It remains to prove that 
$f_{*}\nu$ is also of finite logarithmic energy. This follows from
\begin{align*}
I(f_*\nu) & =\int_{ \Sigma_n\cap V_{z}}\int_{ \Sigma_n\cap V_{z}} \log \frac{1}{|f(z_1)-f(z_2)|}d\nu(z_1)d\nu(z_2)\\
& \leq \log  n + \int_{\Sigma_n\cap V_{z}}\int_{ \Sigma_n\cap V_{z}}\log \frac{1}{|z_1-z_2|}d\nu(z_1)d\nu(z_2)<\infty.
\end{align*}
\end{proof}
 
We will use two specific situations later in the paper: $f$ is the restriction to $K$ of an entire function; and $f$ is the restriction to $K\subset (0,\infty)$ of $f$ holomorphic in the right half plane $H:=\{z\in \C: \Re z >0\}$ with $f(x)>0$ for $x>0$.  These cases are covered in the following two corollaries.

\begin{corollary}  Assume $f$ is holomorphic on a neighborhood of $K$ and the subset $\{z\in{K}: ~f'(z)\neq 0\text{ and }Q(z)<\infty\}$ is nonpolar.
Then there exist $\nu \in \mathcal M(K)$ with $E^Q(\nu)<\infty$.
  
\end{corollary}

\begin{corollary} Let $f:[0,\infty)\to\R$ be a continuous function which is differentiable for $x>0$ and let $K\subset [0,\infty)$. Assume the subset $\{z\in{K}: ~f'(z)\neq 0\text{ and }Q(z)<\infty\}$ is nonpolar.
Then there exist $\nu \in \mathcal M(K)$ with $E^Q(\nu)<\infty$.

\end{corollary}

We state an approximation property that one can use to prove a large deviation result in the unbounded setting. In the next section, we will prove a version for compact sets (Lemma \ref{lemma-scal}) which we will need for our large deviation principle in this case.   
\begin{lemma}\label{approx1}
Let $K$ be a closed and nonpolar subset of $\C$ and let $Q$ be $f-$admissible on $K$. Given $\mu\in\MM(K)$, there exist an increasing sequence of compact sets $K_m$ in $K$ and a sequence of measures $\mu_m\in \MM(K_m)$ such that 
\begin{enumerate}
\item the measures $\mu_m$ tend weakly to $\mu$ as $m\to\infty$; 
\item the energies $E^{Q_{m}}(\mu_m)$ tend to $E^{Q}(\mu)$ as $m\to\infty$, where
$Q_{m}:=Q|_{K_{m}}$.
\end{enumerate}
\end{lemma}

\begin{proof}
Since the measure $\mu$ has finite mass, there exist an increasing sequence of compact subsets $K_{m}$ of $K$ with $\mu(K\setminus K_{m})\leq 1/m$.
Then, the measures $\tilde\mu_{m}:=\mu|_{K_{m}}$ are increasing and tend weakly to $\mu$. Denoting as usual by $k^{+}(x,y)$ and $k^{-}(x,y)$ the positive and negative parts of the function $k(x,y)$ that was defined in (\ref{def-k}), we have, as $m\to\infty$,
$$\chi_m(x,y)k^{+}(x,y)\uparrow k^{+}(x,y) \quad\hbox{and} 
\quad \chi_m(x,y) k^{-}(x,y)\uparrow k^{-}(x,y),$$
$(\mu\times\mu)$-almost everywhere on $K\times K$ where $\chi_m(x,y)$ is the characteristic function of $K_m\times K_m$ and we agree that the left-hand sides vanish when $x=y\notin K_{m}$. By monotone convergence, we deduce that $E^{Q_{m}}(\tilde\mu_{m})$ tend to $E^{Q}(\mu)$ (possibly equal to $+\infty$) as $m\to\infty$, where we recall that the energy $E^{Q}(\mu)$, given by the double integral in (\ref{def-EQ}), is always well defined since $Q$ is $f$-admissible. Setting $\mu_{m}:=\tilde\mu_{m}/\mu(K_{m})$ gives the result.
\end{proof}

 \section{Discretization and additional results for $K$ compact} 
 
 In this section, we restrict to the case where $K$ is {\it compact}. Let $Q\in \AA(K)$ and $w:=e^{-Q}$. Note in this compact setting, the class $\AA(K)$ is universal; i.e., the same for all $f$. Here we naturally assume $f$ is such that there exists $\nu \in \mathcal M(K)$ with $E^Q(\nu)<\infty$ and we discretize the weighted energy problem (\ref{def-EQ}). Let 
  \begin{equation}\label{fvdm} |VDM_k^Q(z_0,...,z_k)|= \ \hbox{{\it weighted Vandermonde of order}} \ k\end{equation}
  $$:=|VDM(z_0,...,z_k)|\exp \Big(-k[Q(z_0)+\cdots + Q(z_k)]\Big)|VDM(f(z_0),...,f(z_k))|$$
  where
$VDM(z_0,...,z_k)=\prod_{0\leq  i<j\leq k}(z_j-z_i)$ and 
  $$\bigl(\delta_k^Q(f)\bigr)(K)=\delta_k^Q(K):=\max_{z_0,...,z_k\in K}  |VDM_k^Q(z_0,...,z_k)|^{2/k(k+1)}.$$
  We will use terminology such as {\it weighted Fekete points}, etc., for notions defined relative to weighted Vandermondes as defined in (\ref{fvdm}). The proofs of Propositions 3.1-3.3 of \cite{VELD} carry over in this setting. 
  
\begin{theorem} \label{sec3} Given $K\subset \C$ compact and not polar, and $Q\in \AA(K)$,
\begin{enumerate}
\item if $\{\mu_k=\frac{1}{k+1}\sum_{j=0}^k\delta_{z_j^{(k)}}\}\subset \mathcal M(K)$ converge weakly to $\mu\in \mathcal M(K)$, then 
\begin{equation}\label{upboundVDM}
\limsup_{k\to \infty} |VDM_k^Q(z_0^{(k)},...,z_k^{(k)})|^{2/k(k+1)}\leq \exp{(-E^Q(\mu))};
\end{equation}
\item we have
$$\delta^Q(K):=\lim_{k\to \infty} \delta_k^Q(K)=\exp{(-E^Q(\mu_{K,Q}))};$$
\item if $\{z_j^{(k)}\}_{j=0,...,k; \ k=2,3,...}\subset K$ and 
\begin{equation}\label{awfp}\lim_{k\to \infty} |VDM_k^Q(z_0^{(k)},...,z_k^{(k)})|^{2/k(k+1)}= \exp{(-E^Q(\mu_{K,Q}))}\end{equation}
then 
$$\mu_k=\frac{1}{k+1}\sum_{j=0}^k\delta_{z_j^{(k)}}\to \mu_{K,Q} \ \hbox{weakly}.$$

\end{enumerate}

\end{theorem}

\begin{proof} We indicate the main ingredients. To prove the analogue of Proposition 3.1 of \cite{VELD}, which is 1. above, we simply observe that for any $M$, 
\begin{align*}
h_M(x,y) & := \min (M,-\log {|x-y|}-\log {|f(x)-f(y)|}) \\
& \leq -\log {|x-y|}-\log {|f(x)-f(y)|}:=h(x,y)
\end{align*}
and $h(x,y)$ is lower semicontinuous if $f$ is continuous. For 2., the analogue of Proposition 3.2 of \cite{VELD}, by uppersemicontinuity of 
$$(z_0,...,z_k)\to |VDM_k^Q(z_0,...,z_k)|,$$
maximizing $(k+1)-$tuples for $\delta_k^Q(K)$ ({\it weighted Fekete points}) exist. Finally, 3., the analogue of Proposition 3.3 of \cite{VELD}, uses the uniqueness of the measure $\mu_{K,Q}$ which minimizes $E^Q$.\end{proof}

\begin{remark} \label{awfr} Arrays $\{z_j^{(k)}\}_{j=0,...,k; \ k=2,3,...}\subset K$ satisfying (\ref{awfp}) will be called {\it asymptotic weighted Fekete arrays} for $K,Q,f$.

\end{remark}

As a last result in this section, we give a refined version of Lemma \ref{approx1} when $K$ is a compact subset of $\C$. This is an analogue of results in 
\cite[Section 5]{LPLD} and will be used in a similar fashion to prove our large deviation result in the compact case. Here $C(K)$ denotes the class of continuous, real-valued functions on $K$.

\begin{lemma}\label{lemma-scal}
Let $K\subset \C$ be compact and nonpolar and let $\mu\in\MM(K)$ with 
$E^{Q}(\mu)<\infty$. There exist an increasing sequence of compact sets $K_m$ in $K$, a sequence of functions $\{Q_m\}\subset C(K)$, and a sequence of measures $\mu_m\in \MM(K_m)$ satisfying
\begin{enumerate}
\item the measures $\mu_m$ tend weakly to $\mu$, as $m\to\infty$; 
\item the energies $I(\mu_m)$ tend to $I(\mu)$ as $m\to\infty$;
\item the energies $I(f_*\mu_m)$ tend to $I(f_*\mu)$ as $m\to\infty$;
\item  the measures $\mu_m$ are equal to the weighted equilibrium measures $\mu_{K,Q_m}$.
\end{enumerate}
\end{lemma}

\begin{proof} 
By Lusin's continuity theorem applied in $K$ and $f(K)$, it is easy to verify that, for every integer $m\geq 1$, there exists a compact subset $K_m$ of $K$ such that $\mu(K\setminus K_m)\leq 1/m$, $p_{\mu}$ is continuous on $K_m$, and $p_{f_{*}\mu}$ is continuous on $f(K_{m})$, respectively considered as functions on $K_m$
and $f(K_{m})$ only. We may assume that $K_m$ is increasing as $m$ tends to infinity. Then, the measures $\tilde\mu_m:=\mu_{|K_m}$ are increasing and tend weakly to $\mu$; similarly the measures $f_*\tilde\mu_m=f_*(\mu_{|K_m})$ are increasing and tend weakly to $f_*\mu$. As in the proof of Lemma \ref{approx1}, we have
$$\chi_m(z,t)\log^{+} {|z-t|}\uparrow \log^{+} {|z-t|} \quad \hbox{and} 
\quad \chi_m(z,t)\log^{+} {|f(z)-f(t)|}\uparrow \log^{+} {|f(z)-f(t)|},$$
as $m\to\infty$, $(\mu\times\mu)$-almost everywhere on $K\times K$ where $\chi_m(z,t)$ is the characteristic function of $K_m\times K_m$ and we agree that the left-hand sides vanish when $z=t\notin K_{m}$. Similar pointwise convergence holds true for the negative parts of the log functions. Hence, by monotone convergence we have
$$I(\tilde\mu_m)\to I(\mu),\quad I(f_*\tilde\mu_m)\to I(f_*\mu),\quad\text{as }m\to\infty,$$
where we observe that the compactness of $K$ implies that the energies $I(\mu)$ and $I(f_{*}\mu)$ are well defined. Indeed, because of the assumption $E^{Q}(\mu)<\infty$, the energies $I(\mu)$ and $I(f_{*}\mu)$ are finite but this is not used here.

Next, define $\mu_m:= \tilde \mu_m/\mu(K_m)$ and for $z\in K$,
$$Q_m(z):= -p_{\mu_m}(z)- p_{f_*\mu_m}(f(z)).$$ 
To show $Q_{m}$ is continuous on $K_{m}$, since $p_{\mu_m}$ and $p_{f_*\mu_m}$ are lower semicontinuous, it suffices to show they are upper semicontinuous. For $p_{\mu_m}$ this follows since $p_{\mu-\mu_{m}}=p_{\mu}-p_{\mu_{m}}$ is upper semicontinuous and $p_{\mu}(z)$ is continuous on $K_{m}$. Similarly, $p_{f_*\mu_m}$ is upper semicontinuous since $p_{f_{*}\mu}-p_{f_{*}\mu_{m}}$ is upper semicontinuous and $p_{f_{*}\mu}(z)$ is continuous on $K_{m}$.

Item 4. follows from the fact that $\mu_{m}$ has compact support with $E^{Q_{m}}(\mu_m)<\infty$ (because $E^{Q}(\mu)<\infty$), and it clearly satisfies the Frostman-type inequalities of Proposition \ref{exun} for $K$ and the weight $Q_m$; hence we have $\mu_m=\mu_{K,Q_m}$. We note that the assumption $E^{Q}(\mu)<\infty$ has only been used to prove 4.
\end{proof}

\section{Bernstein-Walsh inequality and Bernstein-Markov property}

Observe that if we fix all the variables in $VDM_k^Q(z_0,...,z_k)$ in (\ref{fvdm}) except one, say $z_j$, the function $z_j\to VDM_k^Q(z_0,...,z_j,...,z_k)$ is of the form $p_k(z_j)q_k(f(z_j))$
where $p_k, q_k$ are polynomials of degree at most $k$ (we write $p_k, q_k\in \mathcal P_k$). Let $K\subset \C$ be compact and nonpolar. In this section, we prove a Bernstein-Walsh type inequality for functions of the slightly more general form 
$$ h_k(z)=p_k(g(z))q_k(f(z)) \ \hbox{where} \  p_k,q_k\in \mathcal P_k$$
but where we assume $f,g$ are {\it holomorphic} functions on a neighborhood $U$ of $\hat K$, the polynomial hull of $K$. Here, 
$$\hat K=\{z\in \C: |p(z)|\leq ||p||_K \ \hbox{for all} \ p\in \bigcup_k\mathcal P_k\}.$$
In other words, $\hat K$ is the complement of the unbounded component of the complement of $K$. We then utilize this Bernstein-Walsh type inequality in conjunction with a mass density assumption on a finite measure $\nu$ on $K$ to obtain (weighted) Bernstein-Markov properties. The (usual) extremal function of $K$ is defined via
 $$V_K(z):=\sup \{u(z): \ u\leq 0 \ \text{on} \ K \ \text{and} \ u\in\mathcal{L}\}$$
   where 
   $$\mathcal{L}:=\{u(z): \ u \ \text{is subharmonic on} \  \C \ \text{ and} \  u(z)\leq\log^+|z| +C, \ \text{for some} \ C=C(u)\}.$$
For $K$ compact, we have
\begin{equation}\label{forbw}V_K(z):=\sup\{\frac{1}{deg(p)}\log |p(z)| : p\in \cup_k \mathcal P_k, \ ||p||_K\leq 1\}\end{equation}
We let $V_K^*(z):=\limsup_{\zeta \to z}V_K(\zeta)$ denote the upper semicontinuous regularization of $V_K$; thus if $K$ is not polar, $V_K^*$ is subharmonic on
   $\mathbb{C}$, harmonic on $\mathbb{C}\setminus K$ and is, in fact, the Green function with a logarithmic pole at $\infty$ for $\mathbb{C}\setminus K$. We say $K$ is {\it regular} if $V_K^*$ is continuous; equivalently, $V_K=V_K^*$. Note that this is a property of the outer boundary of $K$; i.e., the boundary of the unbounded component of the complement of $K$. The logarithmic capacity of $K$ defined in (\ref{capdef}) can be recovered from $V_K^*$: 
   $$cap(K)= \exp \bigl(-\lim_{|z|\to \infty} [V_K^*(z)-\log |z|]\bigr).$$
   The classical Bernstein-Walsh inequality, coming from (\ref{forbw}), is
   \begin{equation}\label{forbw2}|p_k(z)|\leq ||p_k||_Ke^{kV_K(z)}\end{equation}
   for polynomials $p_k\in \mathcal P_k$.
   
     Given an admissible weight $Q$ on $K$, the weighted Green function for the pair $K,Q$ is $V_{K,Q}^*(z)=\limsup_{\zeta \to z} V_{K,Q}(\zeta)$ where
\begin{equation}\label{forbw3}V_{K,Q}(z):=\sup \{\frac{1}{deg(p)}\log |p(z)|: p\in \cup_k \mathcal P_k, \ ||pe^{-deg(p)Q}||_K\leq 1\}$$
$$=\sup \{u(z):u\in \mathcal{L}, \ u\leq Q \ \hbox{on} \ K\}.\end{equation}
Note that $V_K^*,V_{K,Q}^*\in \mathcal{L}^+$ where
$$\mathcal{L}^+:=\{u \ \hbox{subharmonic in} \ \C: \exists C_1,C_2 \ \hbox{with} \ C_1 +\log^+|z|\leq u(z) \leq C_2 + \log^+|z|\}.$$

 Given $K\subset \C$ compact and $f,g$ holomorphic functions on a neighborhood $U$ of $\hat K$, we now consider functions of the form
 \begin{equation}\label{fn}
 h_k(z)=p_k(g(z))q_k(f(z)), \ z\in U,
 \end{equation}
 where $p_k,q_k\in \mathcal P_k$.
 We denote the collection of such functions by $\mathcal{F}_k$. For $K$ a compact set of the plane we define, for $z\in U$, an 
 extremal function for this class of functions:
 \begin{equation}\label{wkdef}
 W_K(z):=\sup\{\frac{1}{k}\log |h_k(z)|: \ h_k\in\mathcal{F}_k \quad\text{and}\quad ||h_k||_K\leq 1\}.
 \end{equation}
Note that $W_K(z)\leq 0$ for $z\in \hat K$ by the maximum principle. We want to get a Bernstein-Walsh type estimate on functions in $\mathcal F_k$ utilizing $W_K$ valid for a wide class of compact sets $K$. By definition,
\begin{equation} \label{nouse} |h_k(z)|\leq ||h_k||_Ke^{kW_K(z)} \ \hbox{for} \ z\in U,\end{equation}
but this estimate is of no use if $W_K(z)$ is not finite.

In the next four potential theoretic lemmas, we fix $D_A$ to be the closure of a bounded domain in $\C$ where $D_A$ is assumed to be regular and has logarithmic capacity $A$. Then $V_{D_A}^*=V_{D_A}$ and $\lim_{|z|\to \infty}[V_{D_A}(z)-\log |z|] =\log A$. 
 
 \begin{lemma}\label{1}
 Let $D_A$ and $\tau >0$ be given. There is a positive constant $L=L(D_A,\tau )$ such that for all compact subsets $K\subset D_A$ with cap($K)>\tau$, all $k=1,2,...$, and all polynomials $p_k$ of degree $k$ we have
 $$||p_k||_{D_A}\leq e^{kL(D_A,\tau )}||p_k||_K.$$
 \end{lemma}
 \begin{proof}
 Consider the function $V_K^*(z)-V_{D_A}(z)$.
 This function is nonnegative and harmonic on $\overline{\mathbb{C}}\setminus D_A$
 and has value
 $$\log A-\log cap (K)\leq \log A-\log \tau$$ at
 $\infty$. By Harnack's inequality we have, for $z$ with $V_{D_A}(z)=\log 2$, 
 $$V_K^*(z)-V_{D_A}(z)\leq C\log(\frac{A}{\tau})$$
 where $C$ is a constant independent of $K.$
 Thus,$$V_K^*(z)\leq C\log (\frac{A}{\tau})+\log 2$$
 on $\{z\in \C: V_{D_A}(z)=\log 2\}$.
 
 Since $V_K^*$ is subharmonic on $\mathbb{C}$, by the maximum principle the above bound holds on $\partial D_A$. 
  Now, by the usual Bernstein-Walsh inequality (\ref{forbw2}), 
  $$||p_k(z)||_{D_A}\leq||p_k||_Ke^{kL(D_A,\tau )}$$
  where
  $$L(D_A,\tau )=C\log(\frac{A}{\tau}) +\log 2$$
 \end{proof}

 \begin{lemma}\label{2}
 Let $D_A$ and $\tau >0$ be given. Let $K$ be a compact set of $D_A$ with cap($K)>\tau$. Suppose that
 $$K=\cup_{i=1}^s B_i$$
 where the $B_i$ are Borel sets. Then there is a constant $\sigma=\sigma (D_A,\tau ,s )>0$ such that at least one of the sets $B_i$ is of capacity at least $\sigma$.
 
 \end{lemma}
 \begin{proof}
 The proof follows from Theorem 5.1.4(a) of \cite{R}. The constant $\sigma$ depends on the diameter of the bounded set $D_A$.
 \end{proof}

 \begin{lemma}\label{3}
  Let $f$ be holomorphic and nonconstant on a neighborhood of $D_A$. Given $\tau>0$, let $K$ be a subset of $D_A$ such that cap($K)>\tau$. Then there is a constant $\beta =\beta (D_A,\tau,f )>0$ such that
    $${cap}(f(K))\geq \beta.$$ 
 \end{lemma}
 \begin{proof}
 For each point $z_0\in D_A$, there is a neighborhood $V$ of $z_0$ such that the restriction of $f$ to $V,f|_{V}=h^m$, where $h$ is a biholomorphism and $m\in\mathbb{Z}^+.$ Namely if $f'(z_0)\neq 0$ then $f|_{V}$ is a biholomorphism and $m=1$, otherwise $m$ is the least integer such that $f^{(j)}(z_0)\neq 0.$ We may cover $D_A$ by a finite collection of such sets, say $V_i$ for $i=1,2,...,s$ with corresponding positive integers $m_i$. Then we can shrink each set $V_i$ to obtain sets $W_i$ which still cover $D_A$ and such that each $W_i$ has compact closure in $V_i$.
   
   For $J$ a compact subset of a $\overline W_i$ we have
   $$cap(f|_{W_i}(J))\geq C(cap(J))^{m_i}$$
 where for $m_i=1$ we use \cite{R}, Theorem 5.3.1 applied to $(f|_{W_i})^{-1}$ and if $m_i\geq 2$ we use the cited theorem and the fact that under the power map $e_m:z\rightarrow z^m$, Theorem 5.2.5 of \cite{R} gives 
$$[cap(e_m(J))]^{1/m} = cap(e_m^{-1}(e_m(J)))\geq cap(J).$$

 Now $K=\cup_{i=1}^s(K\cap \overline W_i)$ so by Lemma \ref{2} for one of the sets in the union, say $K\cap \overline W_{i_0}$ we have $cap(K\cap\overline W_{i_0})\geq \sigma(D_A,\tau ,s )$ so
 $$cap(f(K))\geq cap (f(\overline W_{i_0}\cap K))\geq Ccap(\overline W_{i_0}\cap K)^{m_{i_0}}\geq C \sigma(D_A,\tau ,s )^{m_{i_0}}.$$
 The constants $C$ which appear above depend only on $f$ and the sets $V_i,W_i$ and not on $K$ so the proof is complete.
 \end{proof}
 
 \begin{remark} Note we do not require $f(K)\subset D_A$ but this assumption will be needed in the next result. \end{remark}

In the upper envelope (\ref{wkdef}) defining $W_K$, given $h_k\in \mathcal{F}_k$ as $h_k(z)=p_k(g(z))q_k(f(z)$, we may multiply $p_k$ by a non-zero scalar $c$ and $q_k$ by $1/c$ without changing $h_k$. We use the following normalization: for $h_k\in \mathcal{F}_k$ and $||h_k||_K=1$ choose a point $z_0\in K$ such that $|h_k(z_0)|=1$. Then multiply $p_k$ and $q_k$ by scalars as above so that $|p_k(g(z_0))|=1$ and $|q_k(f({z_0}))|=1$. The key estimate in this setting is the next result.  
\begin{lemma}\label{l4} Let $f,g$ be holomorphic and nonconstant on a neighborhood of $D_A$ and let $\tau >0$ be given. Let $K$ be a compact subset of $D_A$ such that $f(K), g(K)\subset D_A$ and cap$(K)>\tau$. 
Then there is a constant $M=M(D_A,\tau) >0$ such that for all $k=1,2,...$ and all $h_k \in \mathcal F_k$ normalized as above, 
$$||p_k||_{D_A}\leq e^{Mk}\quad and\quad ||q_k||_{D_A}\leq e^{Mk}.$$
\end{lemma}
\begin{proof}

We give the argument for $q_k$; the one for $p_k$ is similar. We have $cap(f(K))\geq\beta(D_A,\tau,f )>0$ by Lemma \ref{3}. 
Let $\tau'=\sigma(D_A,\beta(D_A,\tau,f) ,2)$ from Lemma \ref{2} and let
$$F_k:=\{t\in f(K): \ q_k(t)\leq e^{M_1k}\}$$ where $M_1$ is to be chosen.
If
\begin{equation}\label{c}
cap(F_k)\geq\tau'
\end{equation}
then by Lemma \ref{1} we have the required estimate on $q_k$:
$$||q_k||_{D_A}\leq ||q_k||_{F_k}e^{kL(D_A,\tau')}\leq 
e^{k(M_1+L(D_A,\tau'))}.
$$
Note here we have used the hypothesis that $f(K)\subset D_A$ to ensure that $F_k\subset D_A$. 

We will show by contradiction that if $M_1$ is sufficiently large then (\ref{c}) must hold.
If (\ref{c}) does not hold then by Lemma \ref{2}
$$cap(G_k)\geq\tau'$$
where $$G_k:=\{t\in f(K): \ q_k(t)\geq e^{M_1k}\}.$$
Now $|p_k(g(z))|\leq e^{-M_1k}$ on $f^{-1}(G_k)\cap K=\{z\in K: \ f(z)\in G_k\}$ since $||h_k||_K=1$ and by \cite{R}, Theorem 5.3.1 
$$cap(f^{-1}(G_k)\cap K)=cap(\{z\in K: \ f(z)\in G_k\})\geq\frac{1}{C} cap(G_k)\geq \tau'/C$$
where $C=\sup||f'||_{D_A}$. But 
$$|p_k(w)|\leq e^{-M_1k} \ \hbox{for} \ w\in g\bigl(f^{-1}(G_k)\cap K\bigr)$$
and Lemma \ref{3} gives
$$cap(g\bigl(f^{-1}(G_k)\cap K\bigr))\geq \beta(D_A,\tau'/C,g) >0.$$
Thus, by Lemma \ref{1}, 
$$||p_k||_{D_A} \leq e^{kL(D_A,\beta(D_A,\tau'/C,g) )}||p_k||_{g\bigl(f^{-1}(G_k)\cap K\bigr)}\leq e^{kL(D_A,\beta(D_A,\tau'/C,g) )}e^{-M_1k}.$$
Here we have used $g(K)\subset D_A$ to insure $g\bigl(f^{-1}(G_k)\cap K\bigr)\subset D_A$. For $M_1$ sufficiently large this contradicts $|p_k(g(z_0))|=1$. 
\end{proof}

We combine the above bounds with the Bernstein-Walsh estimates (\ref{forbw2}) for polynomials and the set $D_A$: for $f,g$ holomorphic on $U \supset D_A$ and $p_k, q_k$ as in Lemma \ref{l4}, i.e., with $h_k\in \mathcal F_k$ normalized so $||h_k||_K=1$, 
$$\frac{1}{k}\log |p_k(g(z))|\leq V_{D_A}(g(z))+M $$
and 
$$\frac{1}{k}\log |q_k(f(z))|\leq V_{D_A}(f(z))+M$$
provided $z \in U$. Note we require
\begin{equation}\label{kcond} K\subset D_A \ \hbox{with} \ f(K), g(K)\subset D_A \ \hbox{and} \ D_A\subset U. \end{equation} 
If $g(z)= z$, this reduces to
\begin{equation}\label{kcond2} K\subset D_A \ \hbox{with} \ f(K) \subset D_A \ \hbox{and} \ D_A\subset U. \end{equation}
We obtain the estimate
$$\frac{1}{k}\log |h_k(z)|\leq 2M+V_{D_A}(g(z)) +V_{D_A}(f(z)), \ z\in U $$ for some constant $M$ for $h_k\in \mathcal F_k$ normalized so $||h_k||_K=1$. Thus the family of subharmonic functions
$$\{\frac{1}{k}\log |h_k(z)|: \ h_k\in\mathcal{F}_k \quad\text{and}\quad ||h_k||_K\leq 1\}$$
is locally bounded above in $U$. This implies that $W_K^*$ is subharmonic on $U$ (see \cite{R}, Theorem 3.4.2) and we have the bound
\begin{equation}\label{BW2}W_K^*(z)\leq 2M+V_{D_A}(g(z)) +V_{D_A}(f(z)), \ z \in U.\end{equation}
This gives a workable Bernstein-Walsh estimate for functions $h_k\in \mathcal{F}_k$, i.e.,
\begin{equation}\label{BW}
|h_k(z)|\leq ||h_k||_Ke^{kW_K^*(z)}, \ z\in U
\end{equation}
with the upper bound (\ref{BW2}) on $W_K^*(z)$. Thus
\begin{equation}\label{BW3}\frac{1}{k}\log \frac{|h_k(z)|}{||h_k||_K} \leq 2M+V_{D_A}(g(z)) +V_{D_A}(f(z)), \ z\in U 
\end{equation}
for all $h_k\in \mathcal F_k$. Note that the right-hand estimates in (\ref{BW2}) and (\ref{BW3}) depend on $D_A$ but the estimates are valid at all $z\in U$ (i.e., at points where $f(z)$ is holomorphic).

We may consider weighted versions of (\ref{wkdef}) and (\ref{BW}). Let $Q\in \AA(K)$. For $z\in U$ we let
\begin{equation}\label{wkqdef} W_{K,Q}(z):=\sup \{\frac{1}{k}\log |h_k(z)|: \  h_k \in \mathcal{F}_k \quad\text{and}\quad ||h_ke^{-kQ}||_K\leq 1\}.\end{equation}
Then $W_{K,Q}\leq Q(z)$ for $z\in K$. Since $\{z\in K: Q(z)< \infty\}$ is not polar, for sufficiently large $C$ the compact set $F:= \{z\in K: Q(z)\leq C\}$ is not polar. Then for $h_k\in \mathcal F_k$ with $||h_ke^{-kQ}||_K\leq 1$ we have $||h_ke^{-kQ}||_F\leq 1$ and
$$||h_k||_F\leq e^{kC}.$$  
From definitions (\ref{wkdef}) and (\ref{wkqdef}), 
$W_{K,Q}(z)\leq W_F^*(z)+C$ for all $z\in F$. Applying (\ref{BW2}) and (\ref{BW3}) with $F$ instead of $K$ (and $M=M(F)$), the family of subharmonic functions defining $W_F^*$ and hence $W_{K,Q}$ is locally bounded above on $U$ and 
$W_{K,Q}^*(z)$ is subharmonic on $U$ with
$W_{K,Q}^*\leq Q(z)$ q.e. on $K$. We get a weighted Bernstein-Walsh estimate for functions $h_k\in \mathcal{F}_k$, namely, from (\ref{wkqdef}),
\begin{equation}\label{BWQ}
|h_k(z)|\leq ||h_ke^{-kQ}||_Ke^{kW_{K,Q}^*(z)}, \ z\in U.
\end{equation}

\begin{remark}\label{IMP} If $f,g:\C \to \C$ are entire, then for any $K\subset \C$ one can find $D_A$ so that the condition (\ref{kcond}) holds; thus the Bernstein-Walsh estimates (\ref{BW}) and (\ref{BWQ}) hold on all of $\C$. Another interesting situation arises taking $f$ and/or $g$ to be branches of power functions $z\to z^{\theta}$ where $\theta >0$. Taking, e.g., $f$ to be a branch defined and holomorphic on $\C\setminus (-\infty,0]$ with $f(z) = |z|^{\theta}$ for $z=|z|>0$, for any $K\subset (0,\infty)$ one can find $D_A\subset H:=\{z\in \C: \Re z >0\}$ so that the condition (\ref{kcond2}) holds. Thus (\ref{BW}) and (\ref{BWQ}) hold on all of $H$. 
\end{remark}

 We next prove a type of regularity of $W_K^*$ in case $K$ is regular. We begin with a lemma. Recall that a compact set $S$ is not thin at a point $\zeta \in S$ if $\limsup_{z\in S\setminus \{\zeta\}}u(z)=u(\zeta)$ for all functions $u$ that are subharmonic in a neighborhood of $\zeta$; otherwise we say $S$ is thin at $\zeta$.

\begin{lemma}
Let $K\subset\mathbb{C}$ be a compact, regular set and let u be a subharmonic function on a neighborhood of $\hat K$. Suppose that $u \leq 0$ q.e. on $K$. Then $u\leq 0$ on $\hat K$.
\end{lemma}
\begin{proof}
Since $u$ is upper semicontinuous, the set $F=\{z\in K: \ u(z)>0\}$ is an $F_{\sigma}$ set. Since $F$ is a polar set it is thin at all points of $\mathbb{C}$ (see \cite{R}, Theorem 3.8.2). But $K$ is not thin at any of its outer boundary points (\cite{R}, Theorem 4.2.4) so $K\setminus F$ is not thin at any outer boundary point of $K.$  This implies that for $\xi$ an outer  boundary point, $u(\xi)=\limsup_{z\in K\setminus F, \ z\rightarrow \xi}u(z)\leq 0$. Then since $u\leq 0$ on the outer boundary by the maximum principle $u\leq 0$ on $\hat K$.
\end{proof}

\begin{corollary}\label{c7}
Let $K\subset \mathbb{C}$ be a compact, regular set satisfying (\ref{kcond}). Then $W_K^*=0$ on $\hat K$.
\end{corollary}

\begin{proof} We have that $W_K^*$ is subharmonic on a neighborhood $U$ of $\hat K$, and $W_K\leq 0$ on $\hat K$. Since $W_K^* = W_K$ q.e., the result follows.
\end{proof}

We define a weighted version of the  Bernstein-Markov inequality for functions in $\mathcal{F}_k.$
\begin{definition}
Given $Q\in \mathcal A(K)$, a Borel measure $\mu$ on $K$ satisfies a weighted Bernstein-Markov inequality for $\mathcal{F}_k$, if given $\epsilon >0$, there is a constant $C$ such that for all $k=1,2,...$ and all $h_k\in\mathcal{F}_k$ we have
\begin{equation}\label{BMWtype2}||h_ke^{-kQ}||_K\leq Ce^{\epsilon k}\int_K |h_k(z)|e^{-kQ(z)}d\mu (z).
\end{equation}
If $\mu$ satisfies a weighted Bernstein-Markov inequality for all continuous $Q$ on $K$, we say $\mu$ satisfies a strong Bernstein-Markov inequality for $\mathcal{F}_k$ on $K$.
\end{definition}

We consider the following mass-density condition for positive Borel measures $\mu$ on $K$: {\sl there exist constants $T,r_0>0$ such that for all $z\in K$}, 
\begin{equation}\label{md}
\mu(D(z,r))\geq r^T \ \hbox{for} \ 0<r\leq r_0.
\end{equation}
Here $D(z,r):=\{w\in \C:|w-z|<r\}$. 
\medskip

We will work with the following class of compact sets:

\begin{definition} \label{sreg} We call a compact set $K$ {\it strongly regular} if every connected component
of $\C\setminus K$ is regular with respect to the Dirichlet
problem. 

\end{definition}

An alternate characterization of strongly regular, given in Lemma \ref{thin} below, is that $K$ is not thin at each of its points. Note that a strongly regular compact set is, indeed, regular; for $K$ is regular precisely when the unbounded component of $\C\setminus K$ is regular with respect to the Dirichlet
problem.  Thus any regular compact set $K$ with connected complement, i.e., $K=\hat K$, is strongly regular. In particular, any regular compact subset of the real line is strongly regular, as is the closure of a bounded domain with $C^1$ boundary. The union of the unit circle with a non-regular compact subset of a smaller circle is regular but not strongly regular. The reason we consider the class of sets in Definition \ref{sreg} is that regularity of a compact set is a property of its outer boundary while, when one considers weighted situations, other points in $K$ can be of influence. Recall for a compact set $K\subset \C$ and a point $z\in K$, we say that Wiener's criterion 
holds at $z$ if
\begin{equation}\label{wiener} \sum_n \frac{n}{\log 1/cap(K\cap S_n)}=\infty \end{equation}
where $S_n=D(z,2^{-n})\setminus D(z,2^{-n-1})$.
Wiener's theorem (cf., \cite{R}, Theorem 5.4.1) states that $K$ is not thin at $z$ precisely when (\ref{wiener}) holds. In particular, if $z$ is a boundary point
of a connected component $G$ of $\C\setminus K$, then
$z$ is a regular boundary point of $G$ with respect to the Dirichlet problem
if and only if Wiener's criterion
holds at $z$. Thus $G$ is regular with respect to the
Dirichlet problem if and only if Wiener's criterion holds at every
boundary point of $G$. This last observation gives the reverse implication of the next result.

\begin{lemma} \label{thin} If $K$ is a compact subset of $\C$ such that every connected component
of $\C\setminus K$ is regular with respect to the Dirichlet
problem, then Wiener's criterion holds at every point of $K$.
\end{lemma}


\begin{proof} Fix $z\in K$; without loss of generality we may assume $z=0$. Since the capacity of the annulus 
$S_n=D(0,2^{-n})\setminus D(0,2^{-n-1})$ is $2^{-n}$, Wiener's criterion 
is certainly true at 0 if 0 is an interior point. Also, by the hypothesis and Wiener's theorem
this criterion holds provided $0$ is a boundary point of a connected component of 
$\C\setminus K$. Thus it is left to verify Wiener's criterion when $0$ is a boundary point
of $K$, but $0$ does not belong to the boundary of any of the components
of $\C\setminus K$.

There are two cases:
\smallskip

1. There are infinitely many $n$ such that for every $r\in [2^{-n-1},2^{n})$
the circle $C(0,r)=\{w: |w|=r\}$ intersects $K$. We consider such
an $n$ and let $w_r\in C(0,r)\cap K$. The mapping $w\to |w|$ is a 
contraction mapping of $K\cap (D(0,2^{-n})\setminus D(0,2^{-n-1}))$
to the interval $[2^{-n-1},2^{-n})$,
which, by assumption, maps onto $[2^{-n-1},2^{-n})$. Since the logarithmic
capacity does not increase under a contraction mapping, and the
capacity of $[2^{-n-1},2^{-n})$ is $2^{-n-1}/4=2^{-n-3}$, we obtain in
this case that $cap(K\cap S_n)\ge2^{-n-3}$, and hence
for this particular $n$ we have
$$\frac{n}{\log 1/cap(K\cap S_n)}\ge \frac{n}{(n+3)\log 2}\ge \frac18.$$
Since this is true for infinitely many $n$, (\ref{wiener}) holds.
\smallskip

2. For all sufficiently large $n$ there is an $r_n\in [2^{-n-1},2^{-n})$
such that $C(0,r_n)$ is disjoint from $K$, i.e., it lies in
a component $G_{r_n}$ of $\C\setminus K$. This $G_{r_n}$ cannot be
the same for infinitely many $n$, for then $0$ would be a
boundary point of that component. Thus, there are
infinitely many $n$ such that $G_{r_n}$ and $G_{r_{n+1}}$ are different.
But then every radial segment $\{re^{it}: r_{n+1}\le r \le r_n\}$ must
intersect $K$, hence the mapping $\{re^{it}\to 2^{-n-2}e^{it}\}$
is a contraction mapping from $K\cap (S_n\cup S_{n+1})$ onto $C(0,2^{-n-2})$.
Therefore,
$$cap(K\cap (S_n\cup S_{n+1}))\ge cap(C(0,2^{-n-2}))=2^{-n-2},$$
and by \cite{R}, Theorem 5.1.4, we have then either
$$\frac{n}{\log 1/cap(K\cap S_n)}\ge \frac{n}{2(n+2)\log 2}\ge \frac16$$
or
$$\frac{n+1}{\log 1/cap(K\cap S_{n+1})}\ge \frac{n+1}{2(n+2)\log 2}\ge \frac16.$$
Thus the series in (\ref{wiener}) contains infinitely
many terms which are at least $1/6$; hence (\ref{wiener}) holds.
\end{proof}


The following result, which is interesting in its own right, will be needed to prove that for strongly regular compact sets, condition (\ref{md}) on $\mu$ implies the strong Bernstein-Markov property.

\begin{lemma}\label{ancona} Let $K$ be a strongly regular compact subset of $\C$. For any $z\in K$ and
$r>0$, there is a regular compact set $L\subset K\cap D(z,r)$ which contains $K\cap D(z,r/2)$. 
\end{lemma}

\begin{proof} For simplicity we take $z=0$ and $r\le 1/2$. By Ancona's theorem \cite{An} the set
$K_n:=K\cap (\overline{D(0,r/2+2^{-n})}\setminus D(0,r/2))$, if nonpolar, contains a regular compact set $F_n$ such that
$$cap \left(\left[K\cap (\overline{D(0,r/2+2^{-n})}\setminus D(0,r/2))\right]\setminus F_n\right)<e^{-n^3}.$$
Setting $F_n=\emptyset$ if $K_n$ is polar, we define 
$$L:=\left(K\cap\overline{D(0,r/2)}\right)\bigcup \left( \cup_n F_n\right).$$
We claim that $L$ is regular. We need to prove that any outer boundary point $z_0$ of $L$ is a regular point. From the 
Wiener criterion (\ref{wiener}), we must show
\begin{equation}\label{e1} \sum_n \frac{n}{\log \bigl(1/cap(L\cap S_n)\bigr)}=\infty \end{equation}
where $S_n=D(z_0,2^{-n})\setminus D(z_0,2^{-n-1})$.

For $z_0\in L$ outside the disk $\overline{D(0,r/2)}$ the union representing $L$ is a locally finite union; thus 
(\ref{e1}) holds by regularity of the sets $F_n$. Also, by the strong regularity of $K$ -- note $L\subset K$ implies $cap(L\cap E)\leq cap(K\cap E)$ for any set $E$ -- (\ref{e1}) is true
for $z_0\in L\cap D(0,r/2)$ (this statement is not necessarily true without the strong regularity hypothesis). It remains to prove (\ref{e1}) for $|z_0|=r/2$. By the strong regularity of 
$K$ we have
$$\sum_n \frac{n}{\log \bigl(1/cap(K\cap S_n)\bigr)}=\infty. $$
Using Theorem 5.1.4 of \cite{R}, since $r\leq 1/2$ it follows that either
\begin{equation}\label{e3} \sum_n \frac{n}{\log \bigl(1/cap(K\cap\overline {D(0,r/2)}\cap S_n)\bigr)}=\infty, \end{equation}
or
\begin{equation}\label{e4} \sum_n \frac{n}{\log \bigl(1/cap(K\cap(D(0,r)\setminus D(0,r/2))\cap S_n)\bigr)}=\infty \end{equation}
(or both). If (\ref{e3}) holds then (\ref{e1}) is true since $L$ contains $K\cap \overline {D(0,r/2)}$,
so assume (\ref{e4}) is true. If ${\cal N}$ is the set of those $n$ for which
$$\frac{n}{\log \bigl(1/cap(K\cap(D(0,r)\setminus D(0,r/2))\cap S_n)\bigr)}>\frac2{n^2},$$
then we still have
\begin{equation}\label{e5} \sum_{n\in {\cal N}} \frac{n}{\log \bigl(1/cap(K\cap(D(0,r)\setminus D(0,r/2))\cap S_n)\bigr)}=\infty. \end{equation}
But again using Theorem 5.1.4 of \cite{R}, 
$$\frac{1}{\log \bigl(1/cap(K\cap(D(0,r)\setminus D(0,r/2))\cap S_n)\bigr)}$$
is no bigger than the sum of 
$$\frac{1}{\log \bigl(1/cap(L\cap(D(0,r)\setminus D(0,r/2))\cap S_n)\bigr)}$$
and
$$\frac{1}{\log \bigl(1/cap((K\setminus L)\cap(D(0,r)\setminus D(0,r/2))\cap S_n)\bigr)}.$$
This latter quantity is no bigger than 
$$\frac{1}{\log \Bigl(1/cap\left(\left[K\cap (\overline{D(0,r/2+2^{-n})}\setminus D(0,r/2))\right]\setminus F_n\right)\Bigr)},$$
which is smaller than $1/n^3$ by our choice of $F_n$. Thus for $n\in {\cal N}$, we necessarily have
$$\frac{n}{\log \bigl(1/cap(L\cap(D(0,r)\setminus D(0,r/2))\cap S_n)\bigr)}> \frac{n/2}{\log \bigl(1/cap(K\cap(D(0,r)\setminus D(0,r/2))\cap S_n)\bigr)}.$$
Hence, in this case, (\ref{e1}) is a consequence of (\ref{e5}).
\end{proof}

\begin{theorem}\label{BM} Suppose $K\subset \C$ is a compact, strongly regular set satisfying (\ref{kcond}) and  $\mu$ is a Borel measure on $K$ satisfying the mass-density condition (\ref{md}). Then $\mu$ is a strong Bernstein-Markov measure for $\mathcal{F}_k$ on $K$.
\end{theorem}
\begin{proof}

Fix $Q\in C(K)$. Given $\epsilon >0$ choose $\delta >0$ such that  $|Q(z_1)-Q(z_2)|\leq\epsilon$ for $z_1,z_2\in K$ with $|z_1-z_2|\leq\delta$ and so that $\{z \in \C: d(z,K)\leq \delta\} \Subset U$ ($d$ being the Euclidean distance). Take a finite collection of disks $\{D(z_j,\delta/4)\}_{j=1,...,m}$ with centers $z_j\in K$ that cover $K$. Since $K$ is strongly regular, by Lemma \ref{ancona}, for each $j$ we can find a regular compact set  $L_j\subset K\cap \overline {D(z_j,\delta/2)}$ with $K\cap \overline {D(z_j,\delta/4)}\subset L_j$. By Corollary \ref{c7}  $W_{L_j}^*$ is continuous on $L_j$ and $W_{L_j}^*=0$ on $L_j$. Thus we can find 
$\sigma =\sigma(\epsilon) > 0$ with $W_{L_j}^*\leq \epsilon$ for all $\zeta$ with $d(\zeta,L_j)\leq \sigma$ for $j=1,...,m$.

Now fix $h_k\in\mathcal{F}_k$ and let $w\in K$ be a point where the function $|h_k(z)|e^{-kQ(z)}$ assumes its maximum on $K$. Then $w\in L_j$ for some $j\in \{1,...,m\}$. For $\zeta \in \overline {D(w,\sigma)}$,
\begin{equation}\label{one}|h_k(\zeta)|\leq ||h_k||_{L_j}e^{k\epsilon}\end{equation}
by the Bernstein-Walsh estimate (\ref{BW}) for $L_j$. On the other hand, by choice of $w$, for any $z\in L_j$,
$$|h_k(z)|e^{-kQ(z)}\leq |h_k(w)|e^{-kQ(w)}$$
(since $L_j\subset K$) and since $|z-w|\leq \delta$ (for $z,w\in \overline {D(z_j,\delta/2)}$), we have $|Q(z)-Q(w) |\leq \epsilon$ so that
$$|h_k(z)|\leq |h_k(w)|e^{k\epsilon}$$
for all $z\in L_j$; i.e., 
\begin{equation}\label{two}||h_k||_{L_j}\leq |h_k(w)|e^{k\epsilon}.\end{equation}
Combining (\ref{one}) and (\ref{two}), we have
$$|h_k(\zeta)|\leq |h_k(w)|e^{2k\epsilon}$$
for all $\zeta \in \overline {D(w,\sigma)}$.

Consider the function
\begin{equation}\label{ut}
U(t):=h_k\Big(w+t\frac{z-w}{|z-w|}\Big).
\end{equation}
Then $t\rightarrow U(t)$ is holomorphic and $U(0)=h_k(w)$ while $U(|z-w|)=h_k(z)$. Also 
\begin{equation}\label{supest}
||U||_{|t|\leq\sigma}\leq |h_k(w)|e^{2k\epsilon}
\end{equation}
and
$$h_k(z)-h_k(w)=U(|z-w|)-U(0)=\int_0^{|z-w|}U'(t)dt.$$
For $z\in D(w,\frac{\sigma }{2})$ we have
$$|h_k(z)-h_k(w)|\leq |z-w|||U'||_{|t|\leq \frac{\sigma }{2}}.$$
Using the Cauchy estimate on $U'$ and (\ref{supest}) we have
$$|h_k(z)-h_k(w)|\leq |z-w|\frac{2}{\sigma}|h_k(w)|e^{2k\epsilon }.$$

Now let $r_k:=e^{-3k\epsilon}$, so
$r_k\leq\frac{\sigma}{4}e^{-2k\epsilon}$ for $k$ large. For $z\in D(w,r_k)$, we have
$$|h_k(z)-h_k(w)|\leq\frac{\sigma}{4}e^{-2k\epsilon}\frac{2}{\sigma }|h_k(w)|e^{2k\epsilon}=\frac{1}{2}|h_k(w)|.$$
So $$|h_k(z)|\geq\frac{1}{2}|h_k(w)|$$
for $z\in D(w,r_k)$ and
$$||h_ke^{-kQ}||_{L^1(\mu)}\geq\int_{K\cap D(w,r_k)}|h_k|e^{-kQ}d\mu$$
$$\geq\frac{1}{2}|h_k(w)|e^{-kQ(w)}e^{-k\epsilon}\mu (D(w,r_k))$$

$$\geq Ce^{-k\epsilon_1}||h_ke^{-kQ}||_K$$ where 
$\epsilon_1=\epsilon (1+3T)$, since for $k$ sufficiently large $\mu (D(w,r_k))\geq r_k^T\geq e^{-3k\epsilon T}.$
\end{proof}

\begin{example} \label{bmex} Some cases where condition (\ref{md}) is satisfied are the following:
\begin{enumerate}
\item $K\subset \R$ is a finite union of compact intervals and $d\mu= dx$, Lebesgue measure; 
\item $K=[0,1]\subset \R$ and $d\mu= x^{\alpha}dx$ where $\alpha >0$;
\item $K\subset \C$ is a fat ($K=\overline {K^o}$) compact set with $C^1$ boundary and $\mu$ is planar Lebesgue measure.
\end{enumerate}

\end{example}

For future use, we generalize the weighted Bernstein-Walsh estimate (\ref{BWQ}) to the unbounded setting in the case where $g(z)=z$. Here, for $K\subset \C$ closed and $Q$ an $f-$admissible weight on $K$ where $f$ is a holomorphic function on a neighborhood $U$ of $K$, the functions in $\mathcal F_k$ are of the form $h_k(z)=p_k(z)q_k(f(z))$ where $p_k,q_k\in \mathcal P_k$. We define $W_{K,Q}$ on $U$ as in (\ref{wkqdef}).

\begin{proposition} \label{bwest} Let $K\subset \C$ be closed and let $f$ be holomorphic on a neighborhood $U$ of $K$. Suppose $Q$ is an $f-$admissible weight on $K$. Let 
\begin{equation}\label{stheta} S=\{z\in K: W_{K,Q}^*(z)\geq Q(z)\}.\end{equation} 
For all $h_k\in \mathcal F_k$, we have
     \begin{equation}\label{BWQmod}
|h_k(z)e^{-kQ(z)}|\leq ||h_ke^{-kQ}||_S \cdot e^{k[W_{K,Q}^*(z)-Q(z)]} \ \hbox{for} \ z\in K.
\end{equation}
\end{proposition}

\begin{proof} Since by definition
$$|h_k(z)|\leq e^{kW_{K,Q}^*(z)}, \ z \in U$$
for $h_k\in \mathcal F_k$ with $||h_ke^{-kQ}||_K=1$, for such $h_k$, 
  \begin{equation}\label{BWQmod2} |h_k(z)e^{-kQ(z)}|\leq e^{k[W_{K,Q}^*(z)-Q(z)]}, \ z\in K.\end{equation}
For $z \in K\setminus S$, from (\ref{BWQmod2}) we have $|h_k(z)e^{-kQ(z)}|<1$ for such $h_k$; hence
$$||h_ke^{-kQ}||_K =  ||h_ke^{-kQ}||_S=1.$$
Inserting this into the right-hand-side of (\ref{BWQmod2}) we have (\ref{BWQmod}) for $h_k\in \mathcal F_k$ normalized so that $||h_ke^{-kQ}||_K=1$. Then (\ref{BWQmod}) follows for all $h_k\in \mathcal F_k$ by normalizing $h_k$. 
\end{proof}

\begin{remark}\label{unbrem} Letting $K_R:=K\cap \{|z|\leq R\}$, if we assume for some $R$ sufficiently large that we have both  $\{z\in K_R: Q(z)<\infty\}$ is nonpolar and $f(K_R) \subset U$, then for such $R$, taking a bounded neighborhood $D_A\subset U$ of the compact set $K_R \cup f(K_R)$, we conclude by the compact case that there is a $M>0$ such that 
  \begin{equation}\label{wkqest}W_{K,Q}^*(z)\leq W_{K_R,Q|_{K_R}}^*(z)\leq 2M+V_{D_A}(z) +V_{D_A}(f(z)), \ z\in U. \end{equation}
Since $K\subset U$, this estimate together with $f-$admissibility of $Q$ imply that the set $S$ in (\ref{stheta}) is compact. In particular, this holds for $f$ as in the two cases described in Remark \ref{IMP}.
\end{remark}

\section{Probabilistic results in compact case} 

In this section, we work with $K$ a compact, nonpolar subset of $\C$ satisfying (\ref{kcond2}) for a fixed $f$ holomorphic on $U$. In this setting, we let $\nu$ be a measure on $K$ with $\nu(K)<\infty$. Fix $Q\in \mathcal A(K)$. Define
\begin{equation}\label{def-Zk}
Z_k:=\int_{K^{k+1}} |VDM_k^Q(z_0,...,z_{k})| d\nu(z_0) \cdots d\nu(z_{k})
\end{equation}
$$=\int_{K^{k+1}} |VDM(z_0,...,z_k)|
e^{-k[Q(z_0)+\cdots + Q(z_k)]}|VDM(f(z_0),...,f(z_k))|
d\nu (z_0)\cdots d\nu (z_k)$$
(recall (\ref{fvdm})). A Bernstein-Markov property (\ref{BMWtype2}) for $\nu$ gives asymptotics of $\{Z_k\}$. 

\begin{proposition} \label{biorthfree} Let $K\subset \C$ be a compact, nonpolar set satisfying (\ref{kcond2}). Suppose $Q\in \mathcal A(K)$ and $\nu$ is a measure on $K$ with $\nu(K)<\infty$ satisfying (\ref{BMWtype2}). Then
\begin{equation}\label{7to4}\lim_{k\to \infty} Z_k^{2/k(k+1)}= \delta^Q(K)=\exp{(-E^Q(\mu_{K,Q}))}.\end{equation}

\end{proposition}

\begin{proof} Since 
$$Z_k\leq \bigl(\max_{z_0,...,z_k\in K} |VDM_k^Q(z_0,...,z_{k})|\bigr) \cdot \nu(K)^{k+1},$$ 
the upper bound
$$\limsup_{k\to \infty} Z_k^{2/k(k+1)}\leq \delta^Q(K)=\exp{(-E^Q(\mu_{K,Q}))}$$
 follows from part 2. of Theorem \ref{sec3}. To prove the lower bound
$$\liminf_{k\to \infty} Z_k^{2/k(k+1)}\geq \delta^Q(K)=\exp{(-E^Q(\mu_{K,Q}))},$$
fix $\epsilon >0$ and a set of weighted Fekete points $(a_0,...,a_k)$  of order $k$ for $K, Q$. Writing
$$|VDM_k^Q(a_0,...,a_k)|=\prod_{i<j}|a_i-a_j|\cdot \prod_{i<j}|f(a_i)-f(a_j)|\cdot \exp \Big(-k[Q(a_0)+\cdots + Q(a_k)]\Big),$$
we recall from the beginning of section 4 that
$$h_k(t):=VDM_k(t,a_1,...,a_k)\cdot VDM_k(f(t),f(a_1),...,f(a_k))\cdot \exp \Big(-k[Q(a_1)+\cdots + Q(a_k)]\Big)$$
$$=VDM_k^Q(t,a_1,...,a_k)=p_k(t)q_k(f(t))\in \mathcal F_k$$
as in (\ref{fn}) with $g(z)=z$ since $p_k, q_k$ are polynomials of degree at most $k$. Then  
$$h_k(t)\exp {\bigl(-kQ(t)\bigr)}$$
attains its maximum modulus on $K$ at $t=f_0$. Applying the weighted Bernstein-Markov type inequality (\ref{BMWtype2}) gives
\begin{equation}\label{forv}|VDM_k^Q(a_0,...,a_k)|\leq Ce^{\epsilon k} \int_K |VDM_k^Q(t,a_1,...,a_k)|d\nu(t).\end{equation}
Now for each fixed $t\in K$, we consider $\tilde h_k(s):=VDM_k^Q(t,s,a_2,...,a_k)\in \mathcal F_k$. Then
$$|VDM_k^Q(t,a_1,...,a_k)|=|\tilde h_k(a_1)|\leq \max_{s\in K} |\tilde h_k(s)|$$
and we apply (\ref{BMWtype2}) in the right-hand-side integral in (\ref{forv}). Continuing this process in each variable, and using (\ref{awfp}) for weighted Fekete points, we obtain the lower bound.
\end{proof}

Given $K\subset \C$ compact, $Q\in \mathcal A(K)$, and a measure $\nu$ on $K$, we define a probability measure $Prob_k$ on $K^{k+1}$: for a Borel set $A\subset K^{k+1}$,
\begin{equation}\label{probk}Prob_k(A):=\frac{1}{Z_k}\cdot \int_A  |VDM_k^Q({\bf X_k})|  d\nu({\bf X_k})
\end{equation}
where ${\bf X_k}=(x_0,...,x_k)$ and $d\nu({\bf X_k})=d\nu(x_0)\cdots d\nu(x_k)$.
Directly from (\ref{7to4}) and (\ref{probk}) we obtain the following estimate.

\begin{corollary} \label{johansson} 
Let $K\subset \C$ be a compact, nonpolar set satisfying (\ref{kcond2}). For  $Q\in \mathcal A(K)$ and $\nu$ a finite measure on $K$ satisfying (\ref{7to4}), given $\eta >0$, define
\begin{equation}\label{def-setA}
 A_{k,\eta}:=\{{\bf X_k}\in K^{k+1}: |VDM_k^Q({\bf X_k})| \geq 
 ( \delta^Q(K) -\eta)^{k(k+1)/2}\}.
\end{equation}
Then there exists $k^*=k^*(\eta)$ such that for all $k>k^*$, 
$$Prob_k(K^{k+1}\setminus A_{k,\eta})\leq 
\Big(1-{\eta}/(2\delta^{Q}(K)) \Big)^{k(k+1)/2}\nu(K^{k+1}). 
$$
\end{corollary}	
	
	We get the induced product probability measure ${\bf P}$ on the space of arrays on $K$, 
	$$\chi:=\{X=\{{\bf X_{k}}\in K^{k}\}_{k\geq 1}\},$$ 
	namely,
	$$(\chi,{\bf P}):=\prod_{k=1}^{\infty}(K^{k+1},Prob_k).$$
As an immediate consequence of Corollary \ref{johansson}, the Borel-Cantelli lemma, and 3. of Theorem \ref{sec3}, we obtain: 

\begin{corollary}\label{416} Let $Q\in \mathcal A(K)$ and $\nu$ a finite measure on $K$ satisfying (\ref{7to4}). For ${\bf P}$-a.e. array $X=\{x_j^{(k)}\}_{j=0,...,k; \ k=2,3,...}\in \chi$, 
$$
\frac{1}{k+1}\sum_{j=0}^k\delta_{x_j^{(k)}}\to \mu_{K,Q} \ \hbox{weakly as }
k\to\infty.$$
\end{corollary}

We remark that 
${\mathcal M}(K)$, with the weak topology, is a Polish space; i.e., a separable, complete metrizable space. A neighborhood basis 
of $\mu \in {\mathcal M}(K)$ is given by sets of the form 
$$G(\mu, k, \epsilon) := \{\sigma  \in {\mathcal M}(K):
|\int_K x^ay^b (d\mu(z) - d\sigma(z) )| < \epsilon$$
$$\hbox{for} \ 0\leq a+b \leq k\} \ \hbox{where} \ z=x+iy.$$

We have all of the ingredients needed to follow the arguments of section 6 of \cite{LPLD} to prove the analogue of Theorem 6.6 there and hence a large deviation principle (Definition \ref{equivform} and Theorem \ref{ldp} below) which quantifies the statement of ${\bf P}$-a.e. convergence for arrays $X=\{x_j^{(k)}\}$ of $\frac{1}{k+1}\sum_{j=0}^k\delta_{x_j^{(k)}}$ to $\mu_{K,Q}$.  Given $G\subset {\mathcal M}(K)$, for each $k=1,2,...$ we let 
\begin{equation}\label{nbhddef}
\tilde G_k:=\{{\bf a} =(a_{0},...,a_{k})\in K^{k+1},~
\frac{1}{k+1}\sum_{j=0}^{{k}} \delta_{a_{j}}\in G\},
\end{equation}
and set
\begin{equation}\label{jkqmu}
J^Q_k(G):=\Big[\int_{\tilde G_{k}}|VDM^Q_k({\bf a})|d\nu ({\bf a})\Big]
^{2/k(k+1)}.
\end{equation}
\begin{definition} \label{jwmuq} For $\mu \in \mathcal M(K)$ we define
$$\overline J^Q(\mu):=\inf_{G \ni \mu} \overline J^Q(G) \ \hbox{where} \ \overline J^Q(G):=\limsup_{k\to \infty} J^Q_k(G);$$
$$\underline J^Q(\mu):=\inf_{G \ni \mu} \underline J^Q(G) \ \hbox{where} \ \underline J^Q(G):=\liminf_{k\to \infty} J^Q_k(G)$$
where the infimum is taken over all open neighborhoods $G\subset \mathcal M(K)$ of $\mu$. If $Q=0$ we simply write $\overline J(\mu), \underline  J(\mu)$.
\end{definition}

Following the steps in section 6 of \cite{LPLD} with Corollary 5.3 there replaced by our approximation result, Lemma \ref{lemma-scal}, we obtain equality of the $ \overline J^{Q}$ and $ \underline J^{Q}$ functionals for any admissible weight $Q$ provided $\nu$ is a strong Bernstein-Markov measure for $\mathcal F_k$ on $K$ (see Theorem 6.6 in \cite{LPLD}). 

 \begin{theorem} \label{rel-J-E} 
 Let $K\subset \C$ be a compact, nonpolar set satisfying (\ref{kcond2}). Let $\nu\in {\mathcal M}(K)$ be a strong Bernstein-Markov measure for ${\mathcal F}_k$ on $K$ (e.g., if $\nu$ satisfies a mass density condition (\ref{md}) and $K$ is strongly regular). \\
 (i) For any $\mu\in \mathcal M(K)$, 
$$
 \log \overline J(\mu)= \log\underline  J(\mu)=-I(\mu)-I(f_*\mu).$$
(ii) Let $Q\in \mathcal A(K)$. Then
$$
 \overline J^{Q}(\mu)=\overline J(\mu)\cdot e^{-2\int_K Qd\mu}
$$
(and with the $\underline J,\underline J^Q$ 
functionals as well) so that,
 \begin{equation}\label{minwtd}\log \overline J^Q(\mu)= \log \underline J^Q(\mu)=-E^{Q}(\mu).
\end{equation}
\end{theorem}

\noindent Thus we simply write $J,J^Q$ without an underline or overline.

Define 
$j_k:  K^{k+1} \to \mathcal M(K)$ via 
\begin{equation}\label{jk} j_k(x_0,...,x_{k})
=\frac{1}{k+1}\sum_{j=0}^{k} \delta_{x_j}.
\end{equation}
The push-forward
$\sigma_k:=(j_k)_*(Prob_k) $ is a probability measure on $\mathcal M(K)$: for a Borel set $G\subset \mathcal M(K)$,
\begin{equation}\label{sigmak}
\sigma_k(G)=\frac{1}{Z_k} \int_{\tilde G_{k}} |VDM_k^Q(x_0,...,x_{k})| d\nu(x_0) \cdots d\nu(x_{k}).
\end{equation}

\begin{definition} \label{equivform}
The sequence $\{\sigma_k\}$ of probability measures on $\mathcal M(K)$ satisfies a {\bf large deviation principle} (LDP) with good rate function $\mathcal I$ and speed $\{s_k\}$ with $s_k\to \infty$ if for all 
measurable sets $\Gamma\subset \mathcal M(K)$, 
\begin{equation}\label{lowb}-\inf_{\mu \in \Gamma^0}\mathcal I(\mu)\leq \liminf_{k\to \infty} \frac{1}{s_k} \log \sigma_k(\Gamma) \ \hbox{and}\end{equation}
\begin{equation}\label{highb} \limsup_{k\to \infty} \frac{1}{s_k} \log \sigma_k(\Gamma)\leq -\inf_{\mu \in \bar \Gamma}\mathcal I(\mu).\end{equation}
\end{definition}

\noindent We will give an interpretation of our specific LDP (Theorem \ref{ldp}) after its statement. On $\mathcal M(K)$, to prove a LDP it suffices to work with a base for the weak topology. The following is a special case of a basic general existence result, Theorem 4.1.11 in \cite{DZ}.

\begin{proposition} \label{dzprop1} Let $\{\sigma_{\epsilon}\}$ be a family of probability measures on $\mathcal M(K)$. Let $\mathcal B$ be a base for the topology of $\mathcal M(K)$. For $\mu\in \mathcal M(K)$ let
$$\mathcal I(\mu):=-\inf_{\{G \in \mathcal B: \mu \in G\}}\bigl(\liminf_{\epsilon \to 0} \epsilon \log \sigma_{\epsilon}(G)\bigr).$$
Suppose for all $\mu\in \mathcal M(K)$,
$$\mathcal I(\mu):=-\inf_{\{G \in \mathcal B: \mu \in G\}}\bigl(\limsup_{\epsilon \to 0} \epsilon \log \sigma_{\epsilon}(G)\bigr).$$
Then $\{\sigma_{\epsilon}\}$ satisfies a LDP with rate function $\mathcal I(\mu)$ and speed $1/\epsilon$. 
\end{proposition}

Following section 7 of \cite{LPLD}, Theorem \ref{rel-J-E} immediately yields a large deviation principle:

\begin{theorem} \label{ldp} Assume $\nu$ is a strong Bernstein-Markov measure for ${\mathcal F}_k$ on $K$, $Q\in \mathcal A(K)$, and $\nu$ satisfies (\ref{7to4}). The sequence $\{\sigma_k=(j_k)_*(Prob_k)\}$ of probability measures on $\mathcal M(K)$ satisfies a large deviation principle with speed $k^{2}/2$ and good rate function $\mathcal I:=\mathcal I_{K,Q}$ where, for $\mu \in \mathcal M(K)$,
\begin{equation*}
\mathcal I(\mu):=\log J^Q(\mu_{K,Q})-\log J^Q(\mu)=E^Q(\mu)-E^Q(\mu_{K,Q}).
\end{equation*}
 \end{theorem}
 
 \noindent Intuitively, this says the following. Given any $\mu\in \mathcal M(K)$ with $\mu \not = \mu_{K,Q}$, we know the probability that a random array $X=\{x_j^{(k)}\}_{j=0,...,k; \ k=2,3,...}\in \chi$ has the property that 
 $\frac{1}{k+1}\sum_{j=0}^k\delta_{x_j^{(k)}}\to \mu$ is zero; the ``rate'' at which the probabilty that this sequence lies in small neighborhoods of $\mu$ tends to zero as $k\to \infty$ like $\exp{\bigl[-k^{2}/2\cdot \mathcal I(\mu)\bigr]}$.

\section {Some results for $K$ unbounded}

In this section, we let $K$ be closed and unbounded; more specifically, recalling Remark \ref{unbrem}, we take 
\begin{equation}\label{fhyp} K\subset \C \ \hbox{for}  \ f \ \hbox{entire or} \
K\subset (0,\infty) \ \hbox{for} \ f \ \hbox{holomorphic in} \ H \ \hbox{with} \ f(x)>0 \ \hbox{if} \ x\in (0,\infty)\end{equation}
where $H$ is the right half plane. We only consider these two situations. We let $Q$ be an $f-$admissible weight on $K$ as in Definition \ref{admit}: the function
$$\psi(x):= Q(x)-\frac{1}{2}\log {[(1+|x|^2)(1+|f(x)|^2)]}$$ 
satisfies $\liminf_{|x|\to \infty, \ x\in K}\psi(x)=\infty$.

We show that Theorem \ref{sec3} remains valid in this setting. Note that the first part of Theorem \ref{sec3},

\begin{enumerate}
\item if $\{\mu_k=\frac{1}{k+1}\sum_{j=0}^k\delta_{z_j^{(k)}}\}\subset \mathcal M(K)$ converge weakly to $\mu\in \mathcal M(K)$, then 
\begin{equation}\label{upboundVDM}
\limsup_{k\to \infty} |VDM_k^Q(z_0^{(k)},...,z_k^{(k)})|^{2/k(k+1)}\leq \exp{(-E^Q(\mu))}
\end{equation}
\end{enumerate}
follows as in the case where $K$ is compact. In order to verify the validity of the rest of Theorem \ref{sec3} in this situation, recall that Proposition \ref{bwest} gives the weighted Bernstein-Walsh estimate (\ref{BWQmod}). Under (\ref{fhyp}), Remark \ref{unbrem} shows that the set $S$ in (\ref{stheta}) is compact. We will use (\ref{BWQmod}) to show that the sequence of probability measures $\mu_k:=\frac{1}{k+1}\sum_{j=0}^k\delta_{z^{(k)}_{j}}$ associated to an array $\{z^{(k)}_{j}\}_{j=0,...,k; \ k=1,2,...}$ of asymptotically weighted Fekete points for $K,Q$ (see (\ref{awfp}) and Remark \ref{awfr}) are {\it uniformly tight}, i.e., given $\epsilon >0$, there exists a compact set $C_{\epsilon}$ such that $\mu_k(K\setminus C_{\epsilon})<\epsilon$ for all $k$. Indeed, we prove a stronger statement. 

\begin{proposition} \label{tight}  For $K$ and $f$ as in (\ref{fhyp}), let $\{z^{(k)}_{j}\}_{j=0,...,k; \ k=1,2,...}$ be an array of asymptotically weighted Fekete points for $K,Q$ where $Q$ is $f-$admissible. Let $S$ be as in (\ref{stheta}). For any $M>0$ with $S \subset D(0,M)$ and any $\delta >0$, there exists $k_0$ such that for all $k>k_0$,
$$\frac{\# \{j:z^{(k)}_{j}\in D(0,M)\}}{k}> 1-\delta.$$ 

\end{proposition}

\begin{proof} Fix $M>0$ with $S \subset D(0,M)$ and $k$. Suppose $|z^{(k)}_{0}|>M$. Now 
$$z^{(k)}_{0}\to VDM^Q_k(z^{(k)}_{0},...,z^{(k)}_{k})=:p_k(z^{(k)}_{0})q_k(f(z^{(k)}_0))e^{-k Q(z^{(k)}_{0})}$$ 
for polynomials $p_k, q_k$ of degree $k$. Let 
$$H_{K,Q}(z):=W_{K,Q}^*(z)-Q(z).$$
By (\ref{BWQmod}),
$$|p_k(z^{(k)}_{0})q_k(f(z^{(k)}_{0}))|e^{-k Q(z^{(k)}_{0})} \leq \bigl(\max_{w\in S} |p_k(w)q_k(f(w))|e^{-kQ(w)}\bigr)\cdot e^{k[H_{K,Q}(z^{(k)}_0)]}$$
$$\leq  \bigl(\max_{w\in S} |p_k(w)q_k(f(w))|e^{-kQ(w)}\bigr)\cdot \rho^k$$
where, by definition of $S$ and $M$,  
$$\rho=\exp [\sup \{k[H_{K,Q}(z)]: z\in K, \ |z|>M\}]<1$$ 
Thus we can find $\tilde z^{(k)}_{0}\in K\cap D(0,M)$ with 
$$|VDM^Q_k(\tilde z^{(k)}_{0},...,z^{(k)}_{k})|=|p_k(\tilde z^{(k)}_{0})q_k(f(\tilde z^{(k)}_{0}))|e^{-k Q(\tilde z^{(k)}_{0})}=\max_{w\in K\cap D(0,M)} |p_k(w)q_k(f(w))|e^{-kQ(w)}$$
so that
$$|VDM^Q_k(\tilde z^{(k)}_{0},...,z^{(k)}_{k})|\geq |VDM^Q_k(z^{(k)}_{0},...,z^{(k)}_{k})|/\rho^k.$$
If ${\# \{j:z^{(k)}_{j}>M\}}/{k}> \delta$, by applying the same reasoning for each such point $z^{(k)}_{j}$, we obtain a set of $k$ points $\tilde z^{(k)}_{0},...,\tilde z^{(k)}_{k} \in K$ where $\lfloor \delta k\rfloor$ of the ``tilde'' points are new and lie in $K\cap D(0,M)$ with
$$|VDM^Q_k(\tilde z^{(k)}_{0},...,\tilde z^{(k)}_{k})|\geq |VDM^Q_k(z^{(k)}_{0},...,z^{(k)}_{k})|/\rho^{\lfloor \delta k\rfloor\cdot k}.$$
Taking $k(k+1)/2$ roots, we get that 
$$\liminf_{k\to \infty} |VDM^Q_k(\tilde z^{(k)}_{0},...,\tilde z^{(k)}_{k})|^{2/k(k+1)}\geq \frac{\lim_{k\to \infty} |VDM^Q_k(z^{(k)}_{0},...,z^{(k)}_{k})|^{2/k(k+1)}}{\rho^{2\delta}}$$
$$=\delta^Q(K)/\rho^{2\delta} >\delta^Q(K),$$
a contradiction.
\end{proof}

The importance of Proposition \ref{tight} is that the rest of Theorem \ref{sec3}; i.e., parts 2. and 3., follows for this setting of $K\subset \C$ unbounded with $Q$ an $f-$admissible weight on $K$. The uniform tightness allows one to extract a subsequence converging in the weak topology on $\mathcal M(K)$; we omit the details.

\begin{corollary}\label{new2.1} For $K$ and $f$ as in (\ref{fhyp}) and $Q$ an $f-$admissible weight on $K$,\begin{enumerate}
\item we have
$$\delta^Q(K):=\lim_{k\to \infty} \delta_k^Q(K)=\exp{(-E^Q(\mu_{K,Q}))};$$
\item if $\{z_j^{(k)}\}_{j=0,...,k; \ k=2,3,...}\subset K$ and 
$$\lim_{k\to \infty} |VDM_k^Q(z_0^{(k)},...,z_k^{(k)})|^{2/k(k+1)}= \exp{(-E^Q(\mu_{K,Q}))}$$
then 
$$\mu_k=\frac{1}{k+1}\sum_{j=0}^k\delta_{z_j^{(k)}}\to \mu_{K,Q} \ \hbox{weakly}.$$

\end{enumerate}
\end{corollary}

In the setting of an unbounded set $K$ and an $f-$admissible weight $Q$, in order to have an analogue of the $\{Z_k=Z_k(\nu)\}$ asymptotics in Proposition \ref{biorthfree} we need some restriction on allowable measures $\nu$ related to $Q$ ensuring finiteness of these quantities. Given a $\sigma$-finite measure $\nu$ on $K$ one can impose the condition that \begin{equation}\label{masscond} \exists \alpha >0,\quad\int_{K}\epsilon(z)^{\alpha}d\nu(z)<\infty\end{equation}
where $\epsilon(z)$ is some nonnegative continuous function that tends to 0 as $|z|$ tends to $\infty$ through points in $K$. For simplicity, we take
\begin{equation}\label{masscond1} -\log \epsilon(z)\leq Q(z)-\log |zf(z)| \end{equation}
where the right-hand-side goes to infinity as $|z|$ tends to $\infty$ through points in $K$ by the $f-$admissibility of $Q$. For such triples $(K,Q,\nu)$, we use the same definition of a weighted Bernstein-Markov type inequality as in Definition \ref{BMWtype2}. 

We note that if $\nu$ satisfies a weighted Bernstein-Markov type inequality on any compact neighborhood of $S$ in (\ref{stheta}), then it satisfies a weighted Bernstein-Markov type inequality on $K$. Combining this observation with the examples given in Example \ref{bmex}, we see that, for appropriate unbounded $K \subset \R$ or $\C$, Lebesgue measure is a strong Bernstein-Markov measure in the setting of (\ref{fhyp}).

Using  (\ref{BWQmod}) in Proposition \ref{bwest} one can prove the analogue of Lemma 8.2 from \cite{LPLD} in our setting.

\begin{lemma} \label{nlemma} For $K$ and $f$ as in (\ref{fhyp}), let $Q$ be $f-$admissible and let $\nu$ be a $\sigma-$finite measure such that $(K,Q,\nu)$ satisfies (\ref{masscond}) and a weighted Bernstein-Markov type inequality. We can find a closed neighborhood $N$ of $S$ (see (\ref{stheta})) and a constant $c>0$ independent of $k$ such that, for all $h_k \in \mathcal F_k$, 
\begin{equation}\label{lemma82}\int_{K}|h_k(z)|e^{-k Q(z)}d\nu(z) \leq(1+\OO(e^{-ck}))\int_{N}|h_k(z)|e^{-k Q(z)}d\nu(z). \end{equation}

\end{lemma}
 
From the lemma, as in section 5, one immediately obtains analogues of free energy asymptotics (Proposition \ref{biorthfree}) and hence Corollary \ref{johansson} and the ${\bf P}$-a.e. convergence result for arrays as in Corollary \ref{416}. A large deviation principle can also be obtained when $Q$ is strongly $f-$admissible; here, Lemma \ref{approx1} can be utilized.

Finally, we remark that using the methods of this paper, many results can be extended to cases where the discrete weighted energy minimization problem (see \ref{fvdm}) involves products of three or more Vandermonde factors.

{\obeylines
\texttt{T. Bloom, bloom@math.toronto.edu
University of Toronto, Toronto, Ontario M5S 2E4 Canada
\medskip
N. Levenberg, nlevenbe@indiana.edu
Indiana University, Bloomington, IN 47405 USA
\medskip
V. Totik, totik@mail.usf.edu
University of South Florida, Tampa, FL 33620 USA
\medskip
F. Wielonsky, wielonsky@cmi.univ-mrs.fr
Universit\'e Aix-Marseille, CMI 39 Rue Joliot Curie
F-13453 Marseille Cedex 20, FRANCE }
}

\end{document}